\documentclass[reqno,10pt,centertags]{amsart}
\usepackage{amsmath,amsthm,amscd,amssymb,latexsym,upref,stmaryrd}
\usepackage{amsfonts,mathrsfs}
\usepackage{hyperref} 

\newcommand*{\mailto}[1]{\href{mailto:#1}{\nolinkurl{#1}}}
\newcommand{\arxiv}[1]{\href{http://arxiv.org/abs/#1}{arXiv:#1}}

\newcommand{\C}{{\mathbb C}}
\newcommand{\N}{{\mathbb N}}
\newcommand{\R}{{\mathbb R}}

\newcommand{\cB}{{\mathcal B}}
\newcommand{\cC}{{\mathcal C}}
\newcommand{\cD}{{\mathcal D}}
\newcommand{\cE}{{\mathcal E}}
\newcommand{\cF}{{\mathcal F}}
\newcommand{\cH}{{\mathcal H}}

\newcommand{\e}{{\varepsilon}}

\newcommand{\si}{\sigma}

\newcommand{\st}{\;|\;}
\newcommand{\bs}{\backslash}

\DeclareMathOperator{\dist}{dist}
\DeclareMathOperator{\supp}{supp}

\DeclareMathOperator{\ran}{ran}
\DeclareMathOperator{\dom}{dom}

\DeclareMathOperator{\Def}{def}

\DeclareMathOperator*{\slim}{s-lim}

\newcommand{\loc}{\text{\rm{loc}}}

\newcommand{\Om}{\Omega}

\newcommand{\no}{\notag}
\newcommand{\lb}{\label}
\newcommand{\f}{\frac}

\newcommand{\ol}{\overline}

\newcommand{\wti}{\widetilde}

\newcommand{\bi}{\bibitem}

\renewcommand{\ge}{\geqslant}
\renewcommand{\le}{\leqslant}

\let\geq\geqslant
\let\leq\leqslant

\allowdisplaybreaks 
\numberwithin{equation}{section}

\newtheorem{theorem}{Theorem}[section]

\newtheorem{lemma}[theorem]{Lemma}
\newtheorem{corollary}[theorem]{Corollary}
\newtheorem{definition}[theorem]{Definition}
\newtheorem{hypothesis}[theorem]{Hypothesis}
\newtheorem{example}[theorem]{Example}
\theoremstyle{remark}
\newtheorem{remark}[theorem]{Remark}

\begin{document}

\title[Decoupling of Deficiency Indices]{Decoupling of Deficiency 
Indices and Applications to Schr\"odinger-Type Operators with Possibly Strongly 
Singular Potentials}

\author[F.\ Gesztesy]{Fritz Gesztesy}
\address{Department of Mathematics,
University of Missouri, Columbia, MO 65211, USA}
\email{\mailto{gesztesyf@missouri.edu}}
\urladdr{\url{https://www.math.missouri.edu/people/gesztesy}}
%\urladdr{https://www.math.missouri.edu/people/gesztesy}
\address{Present address: Department of Mathematics, 
Baylor University, One Bear Place \#97328,
Waco, TX 76798-7328, USA}
\email{Fritz$\_$Gesztesy@baylor.edu}

\author[M.\ Mitrea]{Marius Mitrea}
\address{Department of Mathematics,
University of Missouri, Columbia, MO 65211, USA}
\email{\mailto{mitream@missouri.edu}}
\urladdr{\url{https://www.math.missouri.edu/people/mitream}}

\author[I.\ Nenciu]{Irina Nenciu}
\address{Department of Mathematics, University of Illinois at Chicago, Chicago, 
IL, USA, and Institute of 
Mathematics ``Simon Stoilow'' of the Romanian Academy, Bucharest, Romania}
\email{\mailto{nenciu@uic.edu}}
\urladdr{\url{http://www.math.uic.edu/~nenciu}}

\author[G.\ Teschl]{Gerald Teschl}
\address{Faculty of Mathematics\\ University of Vienna\\
Oskar-Morgenstern-Platz 1\\ 1090 Wien\\ Austria\\ and International 
Erwin Schr\"odinger Institute for Mathematical Physics\\ 
Boltzmanngasse 9\\ 1090 Wien\\ Austria}
\email{\mailto{Gerald.Teschl@univie.ac.at}}
\urladdr{\url{http://www.mat.univie.ac.at/~gerald/}}

%\dedicatory{}
\thanks{Work of M.\,M.\ was partially supported by the Simons Foundation Grant $\#$\,
281566. I.\,N.'s work was partially supported by the National Science Foundation (NSF) Grant 
DMS--1150427. 
G.\,T.\ was supported by the Austrian Science Fund (FWF) under Grant No.~Y330.}
\thanks{Adv. Math. {\bf 301}, 1022--1061 (2016)}
\date{\today}

%%%%%%%%%%%%%%%%%%%%%%%%%%%%%%%%%%%%%%
%%%%%%%%%%%%%%%%%%%%%%%%%%%%%%%%%%%%%%
\begin{abstract}
We investigate closed, symmetric $L^2(\mathbb{R}^n)$-realizations $H$ of Schr\"o\-dinger-type 
operators $(- \Delta +V)\upharpoonright_{C_0^{\infty}(\mathbb{R}^n \setminus \Sigma)}$ whose
potential coefficient $V$ has a countable number of well-separated singularities on 
compact sets $\Sigma_j$, $j \in J$, of $n$-dimensional Lebesgue measure zero, with 
$J \subseteq \mathbb{N}$ an index 
set and $\Sigma = \bigcup_{j \in J} \Sigma_j$. We show that the defect, $\mathrm{def}(H)$, 
of $H$ can be computed in terms of the 
individual defects, $\mathrm{def}(H_j)$, of closed, symmetric $L^2(\mathbb{R}^n)$-realizations of 
$(- \Delta + V_j)\upharpoonright_{C_0^{\infty}(\mathbb{R}^n \setminus \Sigma_j)}$ with potential 
coefficient $V_j$ localized around the singularity $\Sigma_j$, $j \in J$, where 
$V = \sum_{j \in J} V_j$. In particular, we prove 
\[
\mathrm{def}(H) = \sum_{j \in J} \mathrm{def}(H_j), 
\]
including the possibility that one, and hence both sides equal $\infty$. 
We first develop an abstract approach to the question of decoupling of deficiency 
indices and then apply it to the concrete case of Schr\"odinger-type operators in 
$L^2(\mathbb{R}^n)$.

Moreover, we also show how operator (and form) bounds for $V$ relative to 
$H_0= - \Delta\upharpoonright_{H^2(\mathbb{R}^n)}$ can be estimated in terms of the operator (and form) bounds 
of $V_j$, $j \in J$, relative to $H_0$. Again, we first prove an abstract result and then show its applicability to Schr\"odinger-type operators in $L^2(\mathbb{R}^n)$. 

Extensions to second-order (locally uniformly) elliptic differential operators on $\mathbb{R}^n$ with 
a possibly strongly singular potential coefficient are treated as well.  
\end{abstract}
%%%%%%%%%%%%%%%%%%%%%%%%%%%%%%%%%%%%%%
%%%%%%%%%%%%%%%%%%%%%%%%%%%%%%%%%%%%%%

\maketitle

{\scriptsize{\tableofcontents}}

%%%%%%%%%%%%%%%%%%%%%%%%%%%%%%%%%%%%%%
%%%%%%%%%%%%%%%%%%%%%%%%%%%%%%%%%%%%%%
\section{Introduction}  \lb{s1}
%%%%%%%%%%%%%%%%%%%%%%%%%%%%%%%%%%%%%%
%%%%%%%%%%%%%%%%%%%%%%%%%%%%%%%%%%%%%%

The main theme of this paper is centered around the question of (essential) self-adjointness
of Schr\"odinger-type operators $H$ in $L^2(\R^n)$. Specifically, we are interested in the case
where the potential $V$ has a (potentially infinite) number of well-separated singularities, and
we want to know how each of these singularities contributes to the
deficiency index of the full operator $H$. We were inspired to work on
this question, in particular, by the paper of Felli, Marchini,
and Terracini \cite{FMT07}, who consider the question
of essential self-adjointness of a specific Schr\"odinger operator
with so-called multipolar inverse-square potentials. Unlike
in their work, we consider a more general set-up, and proceed by using 
more abstract localization techniques. More precisely,
we consider general compact subsets $\Sigma_j \subset \R^n$ of $n$-dimensional 
Lebesgue measure zero as the singularities
of potentials $V_j$, $j \in J$ ($J \subseteq \N$ a suitable index set), respectively, 
and prove the localization
of deficiency indices solely under the additional requirement that
the singular sets $\Sigma_j$ be uniformly separated by some distance $\varepsilon > 0$. 
In particular, one can then directly conclude that if every Schr\"odinger operator $H_j$ in 
$L^2(\R^n)$ with localized potential $V_j$ around the singularity $\Sigma_j$ is essentially 
self-adjoint, then so is the full Schr\"odinger operator $H$ associated with the singularity set 
$\bigcup_{j \in J} \Sigma_j$.

While we will discuss in Section \ref{s5} several applications of our results, 
one easily presentable consequence is the following fact: 

%%%%%%%%
\begin{theorem}\label{t1.1} 
Let $J\subseteq \N$ be an index set, and $\{x_j\}_{j\in J}\subset\R^n$ be a 
set of points such that
\begin{equation}
\inf_{\substack{j, j' \in J \\ j \neq j'}} |x_j - x_{j'}| > 0.    
\end{equation} 
Fix $\delta>0$, 
and consider the multipolar inverse-square potential function:
\begin{equation}\label{E:defnV_multipole}
V(x)=V_0(x) + \sum_{j\in J} V_j(|x-x_j|)\chi_{B_n(x_j;\delta)}(x) \, \text{ for all } \, 
x\in\R^n \backslash \{x_j\}_{j\in J},
\end{equation}
where $B_n(x_j;\delta)$ denotes the $($open$)$ ball in $\R^n$ centered at
$x_j$ of radius $\delta$, $V_0\in L^\infty(\R^n)$, and $V_j$ satisfies 
\begin{equation}
V_j(r)= \frac{c_j}{r^2} + \wti{V}_j(r), \quad c_j\in\R, \; r\wti{V}_j(r) \in L^1((0,\delta))\cap L^\infty_{\loc}((0,\delta]), \quad j \in J.   \lb{1.2} 
\end{equation}
Then the Schr\"odinger operator 
\begin{equation}
H=-\Delta+V, \quad \dom(H)=C_0^\infty(\R^n \backslash \{x_j\}_{j\in J})    \lb{1.3} 
\end{equation} 
is essentially self-adjoint if and only if 
\begin{equation}\label{E:cond_c_j}
c_j\geq 1-\left(\frac{n-2}{2}\right)^2 = -\frac{n(n-4)}{4} \, \text{ for every } \, j\in J.
\end{equation}
\end{theorem}
%%%%%%%%%

In fact, in the above theorem one could replace each $V_j(|x-x_j|)$ by arbitrary potentials (not necessarily radial) which are locally bounded away
from $x_j$ such that 
\begin{equation} 
H_j=-\Delta + V_j \chi_{B_n(x_j;\delta)}, \quad \dom(H_j) = C_0^{\infty}(\R^n\backslash\{x_j\}), 
\quad j \in J,   \lb{1.4} 
\end{equation} 
 is essentially self-adjoint in $L^2(\R^n)$. For example, one still gets essential self-adjointness
of $H_j$ if we require $V_j(x) \ge -n(n-4)/(4|x-x_j|^2)$, $x \in \R^n\backslash\{x_j\}$, and hence 
essential self-adjointness of $H$ if the latter inequality holds for all $j \in J$. 
The only reason for our specific choice \eqref{1.2} is that in this case the essential self-adjointness issue is well-understood (cf.\ \cite[Theorem~X.30]{RS75}). Indeed, another natural choice is a  dipole potential of the type, 
\begin{equation}
V_j(x) = \f{(x - x_j) \cdot d_j}{|x-x_j|^3}, \quad x \in \R^n \backslash \{x_j\},  
\end{equation}
for fixed $d_j \in \R^n$ (cf.\ \cite{FMT08}, \cite{FMT09}, \cite{HG80}), or a van der Waals-type 
potential of the form,  
\begin{equation}
V_j(x) = \f{A_j}{|x - x_j|^m} - \f{B_j}{|x- x_j|^{6}}, \quad x \in \R^n \backslash \{x_j\},
\end{equation}
for fixed $A_j \in (0,\infty)$, $B_j \in \R$ (although, physics requires $B_j \in (0,\infty)$), 
and $m \geq 10$, with $m=12$ giving rise to a 
Lennard--Jones-type potential (cf., e.g., \cite[pp.~53--59]{Ki05}, \cite[Sect.~3.2]{Lu78}).  
Moreover, the exceptional one-point sets $\{x_j\}$ in such examples can be generalized to compact subsets, $\Sigma_j \subset \R^n$, of $n$-dimensional Lebesgue measure zero, uniformly separated by a positive distance. We also note that our hypotheses on the discrete set $J$ are sufficiently general to describe periodic structures (crystals), half-crystals, an infinite graphene sheet, etc. 

More generally, and far beyond Theorem \ref{t1.1}, we shall derive results to the effect that 
the defect of operators $H$ of the form \eqref{1.3} can actually be expressed as the sum over the individual defects of the operators $H_j$ of the form \eqref{1.4}, 
\begin{equation}
\Def(H) = \sum_{j \in J} \Def(H_j),    \lb{1.5}
\end{equation} 
including the possibility that one, and hence both sides of \eqref{1.5} equal $\infty$. Here
\begin{equation}
\Def(A) = \big[n_+(A) + n_-(A) \big)]/2,   \lb{1.6} 
\end{equation}
with $n_{\pm}(A)$ the standard deficiency indices of the symmetric operator $A$ in a Hilbert 
space $\cH$ (cf.\ Definition \ref{d2.4} for more details). In particular, for $A$ bounded from below, 
or $A$ commuting with a conjugation (cf.\ Definition \ref{d2.7}), $\Def(A) = n_{\pm}(A)$.
In fact, this extension of the case of deficiency indices zero to general deficiency indices is one of our principal results.
For further details we refer to Sections \ref{s4} and \ref{s5}. At this point, however, we decided to focus on how Theorem \ref{t1.1} generalizes and compares to Theorem~8.4 of \cite{FMT07}, which is a special case of our result:

$\bullet$ In contrast to our result, [20, Theorem 8.4] considers the special case 
$\widetilde{V}_j=0$ for every $j\in J$.

$\bullet$ Most significantly, we are able to completely eliminate their conditions (see (19) in \cite[Lemma 3.5]{FMT07}) that:
\begin{equation}\label{E:FMT07_cond1}
\sum_{j\in J} |x_j|^{-(n-2)}<\infty
\end{equation}
and that there exists a constant $C \in (0,\infty)$ such that
\begin{equation}\label{E:FMT07_cond2}
\sum_{j\in J \setminus\{j'\}} |x_j-x_{j'}|^{-(n-2)}\leq C, \quad j'\in J.
\end{equation}
(Note that we have slightly reformulated the assumptions from \cite{FMT07} in order to
fit our notation, as well as the fact that our index set $J$ is not necessarily equal
to $\N$.) In particular (as observed in \cite[Remark 3.7]{FMT07}), 
if we wish to place the singularities
$\{x_j\}_{j\in J}$ on a lattice $\mathbb Z^d\times 0\subseteq \R^n$, $d\leq n$, 
the two conditions \eqref{E:FMT07_cond1} and \eqref{E:FMT07_cond2} above 
only allow for $d<n-2$. By contrast, our Theorem \ref{t1.1}  poses 
no restrictions on the dimension $d$ of the lattice.

$\bullet$ In \cite{FMT07}, conditions \eqref{E:FMT07_cond1} and \eqref{E:FMT07_cond2} are 
used in a technical result, yielding the existence of a certain radius $\delta>0$,
which will be their radius of localization around each singularity (what the authors
call shattering of reticular singularities). This specific
radius ensures that the potential $V$, defined as in \eqref{E:defnV_multipole},
but with $c_j=-\lambda> -\left(\frac{n-2}{2}\right)^2$ for all $j\in J$, 
satisfies a certain minimization condition (see \cite[Lemma 3.5]{FMT07}). This
then allows the authors to prove their result (see \cite[Theorem 8.4]{FMT07}).
Since we do not need to impose the two restrictions \eqref{E:FMT07_cond1}, \eqref{E:FMT07_cond2}
on the locations of our singularities, we are free to choose the radius $\delta>0$ 
arbitrarily. 

$\bullet$ Finally, we note that we prove our result with distinct coupling constants
$c_j\in \R$ in each singular potential $V_j$, and show that in order to
guarantee essential self-adjointness of $H$, all of these coupling constants must
satisfy the well-known condition \eqref{E:cond_c_j}.

The role played by the radius $\delta>0$ in our Theorem \ref{t1.1}  is
different from that played in \cite[Theorem 8.4]{FMT07}. In our case, the localization
by the characteristic function $\chi_{B_n(x_j;\delta)}$ takes place
only in order to ensure that the sum defining the potential $V$ in 
\eqref{E:defnV_multipole} is finite at every $x\in\R^n \backslash \{x_j\}_{j\in J}$.
In \cite[Theorem 8.4]{FMT07}, this is achieved by requiring $\delta<1/2$, and
without loss of generality we could have also required that $\delta<\varepsilon_0/2$,
with $\varepsilon_0=\inf_{j\neq j'\in J} \big|x_j-x_{j'}\big| > 0$;
in both cases, these conditions would imply that the local potentials
$c_j\chi_{B_n(x_j;\delta)}(x)/|x-x_j|^2$ have relatively disjoint supports. However,
we know from \cite[Cor.~2]{BF77} that the deficiency indices (and hence the property of being essentially self-adjoint) of a Schr\"odinger operator remain invariant under the addition of relatively
bounded potentials with relative bound less than one. This in fact allows one to consider any 
potentials $V_j$ with singularity at $x_j$, as long as their sum $V=\sum_{j\in J} V_j$ is finite
at every $x\in\R^n \backslash \{x_j\}_{j\in J}$. The statement of Theorem \ref{t1.1} 
is only an example of such a phenomenon, and its proof (carried out at the end of 
Section \ref{s5}) illustrates the general method.

While reference \cite{FMT07} by Felli, Marchini, and Terracini is closest in spirit to our work, 
and motivated ours, there are many related papers on this subject which we briefly 
turn to next. This subject has a long history, especially in connection with essential 
self-adjointness of Schr\"odinger and Dirac-type operators, that is, in the special case of 
vanishing deficiency indices. Indeed, the literature devoted to essential self-adjointness of 
Schr\"odinger and Dirac-type operators with possibly strongly singular potential coefficients is enormous and no exhaustive list of references can possibly be included here. Instead, we 
refer to a few representative and classical items and the references cited therein, such as, 
\cite{Be85}, \cite{BF78}--\cite{Br04}, \cite{CG78}, \cite{Cy82}, \cite{De77}, \cite{EEM76}, \cite{FF14}, \cite{FDR74}, \cite{Fr77}, \cite{IK62}, \cite{Jo64}--\cite{Ka74}, \cite{Ka81}, \cite{Ke79}, 
\cite{Kl80}--\cite{Kn77}, \cite{Ma81}, \cite{Ne76}--\cite{Ol99}, \cite[Ch.~X]{RS75}, \cite{Ro74}, \cite{Sc75}, \cite[Ch.~9]{Sc86}, \cite{Sc72}, \cite{Sh01}--\cite{St56}, \cite{Wa77}--\cite{Wa69}, \cite{Wi58}. 
The case of multi-center singularities and the associated decoupling of deficiency indices in the sense of relation \eqref{1.5} also have a longer history and go back to sources such as \cite{Be85}, \cite{BF77}, \cite{BF78}, \cite{BG85}, \cite{HG80}, \cite{Ka82}, \cite{Ka85}, \cite{Kl80}, \cite{Mo79}, \cite{Ne77}. For more recent activities in connection with boundedness from below and essential self-adjointness of Schr\"odinger-type operators with multi-center singularities we refer to \cite{Ca14}, \cite{CS14}, \cite{FMT07}, \cite{FMT08}, \cite{FMT09}.  

Next, we turn to the organization of this paper. 
In Section \ref{s2} we briefly present the functional analytic background of our
work; this is very well-known material, but we include it for clarity of notation,
and for ease of reference. Section \ref{s3} presents an abstract form of several results
of Morgan \cite{Mo79}, which deal with the question of localization of
relative form and operator boundedness for infinite sums of operators. We note  
that the form results in Section \ref{s3} apply to the strong $|x - x_j|^{-2}$-type singularities 
discussed in Theorem \ref{t1.1}  and yield relative (form) boundedness (cf.\ the comments 
surrounding \eqref{3.3}--\eqref{3.20}). To prove the equivalence in Theorem \ref{t1.1}, we 
first tackle the full problem of 
decoupling of deficiency indices for infinite sums of operators in Section \ref{s4} from an abstract 
point of view. The main technical and abstract result of our paper, Theorem \ref{t2.10}, is 
presented and proved in this section. We emphasize that the hypotheses of Theorem \ref{t2.10} 
can be checked in practice in a fairly straightforward manner as shown in Section \ref{s5}, where we apply our general decoupling result 
to both Schr\"odinger and second-order elliptic differential operators with potential coefficients 
that can exhibit countably infinitely many (uniformly separated) singularities. The main result of this section is 
Theorem \ref{t4.9}, which shows that deficiency indices localize around singularities under extremely general conditions, see Hypothesis \ref{h4.2}. More concrete examples,
as well as the proof of Theorem \ref{t1.1}, follow then as
a simple consequence. We extend this localization result to more
general second-order elliptic operators in Theorem \ref{t4.11}. We close our paper
with an appendix containing some background on the notion of support
for arbitrary functions on arbitrary subsets in $\R^n$. This is needed in our case
in order to deal with the fact that our localization results in Section \ref{s5} permit ``arbitrarily bad'' compact singularity sets $\Sigma_j$, $j \in J$.

Finally, we briefly summarize some of the notation used in this paper: Throughout, we 
denote by $\cH$ 
a separable complex Hilbert space, by $\|\cdot \|_{\cH}$ the norm in $\cH$, by 
$(\cdot,\cdot)_{\cH}$ the scalar product 
in $\cH$ (linear in the second argument), and by $I_{\cH}$ the identity operator in $\cH$. 

Next, if $T$ is a linear operator mapping (a subspace of) a Hilbert space into another, then 
$\dom(T)$ and $\ker(T)$ denote the domain and kernel (i.e., null space) of $T$. 
The closure of a closable operator $S$ is denoted by $\ol S$. 
%The spectrum and resolvent set of a closed linear operator in a Hilbert space will be denoted 
%by $\sigma(\cdot)$ and $\rho(\cdot)$, respectively. 

The Banach space of bounded linear operators on the separable complex Hilbert space 
$\cH$ is denoted by $\cB(\cH)$. 

The symbol $\dotplus$ denotes the direct sum in the sense of Banach spaces (to be 
distinguished from the orthogonal direct sum $\oplus$ in $\cH$). 

If $J \subseteq \N$ denotes an index set, we denote by $\# (J)$ the cardinality of $J$. 
We also employ the notation $\N_0 = \N \cup \{0\}$. 

We employ the usual multi-index notation for partial derivatives of functions on $\R^n$, that is, 
$\partial^{\alpha} = \partial_{x_1}^{\alpha_1} \cdots \partial_{x_n}^{\alpha_n}$, 
$\alpha = (\alpha_1,\dots,\alpha_n) \in \N_0^n$, $|\alpha| = \alpha_1 + \cdots + \alpha_n$, etc. Similarly, 
we will employ the notation $\partial_k = \partial/\partial_{x_k}$, $1\leq k \leq n$. 
The symbol $B_n(x;r) \subset \R^n$ represents the open ball of center $x \in \R^n$ and 
radius $r>0$.

For $\Omega \subseteq \R^n$ open, $C_0^{\infty}(\Omega)$ denotes the set of 
$C^{\infty}$-functions with compact support contained in $\Omega$, $\cD(\Omega)$ represents 
the corresponding space of test functions obtained via an inductive limit procedure, and 
$\cD'(\Omega)$ denotes the corresponding space of distributions (i.e., continuous linear 
functionals on $\cD(\Omega)$). 

For simplicity, if the underlying Lebesgue measure $d^n x$ is understood, we abbreviate 
$L^p(\R^n) := L^p (\R^n; d^n x)$ and $L^p_{\loc}(\R^n) := L^p_{\loc} (\R^n; d^n x)$. 
Similarly, given a locally integrable 
weight $0 < w \in L^1_{\loc} (\R^n)$, we will employ the notation 
$L^p_w (\R^n) := L^p (\R^n; w d^n x)$, $p \geq 1$. For $\Omega \subseteq \R^n$ 
open, standard $L^2$-based Sobolev spaces are denoted as usually by 
$H^{k}(\Omega)$, and analogously for their local versions, $H^{k}_{\loc}(\Omega)$, 
$k \in \N$. 

The closure of a set $M \subset \R^n$ will be denoted by $\ol M$,  $\mathring{M}$ denotes 
the interior of $M$, and the symbol 
$\chi_M$ is used for the characteristic function of the set $M\subset \R^n$.

%%%%%%%%%%%%%%%%%%%%%%%%%%%%%%%%%%%%%%
%%%%%%%%%%%%%%%%%%%%%%%%%%%%%%%%%%%%%%
\section{Functional Analytic Background}  \lb{s2}
%%%%%%%%%%%%%%%%%%%%%%%%%%%%%%%%%%%%%%
%%%%%%%%%%%%%%%%%%%%%%%%%%%%%%%%%%%%%%

We start with a bit of background (cf., e.g., \cite[Ch.\ 8]{AG81}, \cite[Part III]{Fa75}, 
\cite[Sect.~V.3]{Ka80}, \cite[Sect.\ X.1]{RS75}, \cite[Part~VI]{Sc12}, 
\cite[Ch.\ 2]{Te14}, \cite[Sect.\ 5.4, Ch.\ 8]{We80} for details).  

Let $A$ be a densely defined operator in $\cH$, then $A$ is called {\it symmetric} if 
$A \subseteq A^*$. Thus, $A^*$ is also densely defined, rendering 
$A$ closable with $\ol A = (A^*)^*$. In addition, $A$ is {\it self-adjoint} if $A = A^*$, and  
{\it essentially self-adjoint} if $\ol A$ is self-adjoint. 

%%%%%%%%%%
\begin{theorem}[Basic criterion for self-adjointness] \lb{t2.1}
Suppose $A$ is symmetric in $\cH$. Then the following statements $(i)$--$(iii)$ are equivalent: 
\\[.5mm] 
$(i)$ $A$ is self-adjoint. \\[.5mm] 
$(ii)$ $A$ is closed and $\ker(A^*\pm iI)=\{0\}$. \\[.5mm] 
$(iii)$ $\ran \, (A \pm iI)=\cH$.
\end{theorem}
%%%%%%%%%%

%%%%%%%%%% 
\begin{theorem} \lb{t2.2}
Suppose $A$ is symmetric and closed in $\cH$. Then the following assertions hold: \\[.5mm]  
$(i)$ $\dim \, (\ker(A^*-zI))$ is independent of $z\in\C_+=\{z\in\C\st \Im(z)>0\}$. \\[.5mm] 
$(ii)$ $\dim \, (\ker(A^*-zI))$ is independent of $z\in\C_-=\{z\in\C\st \Im(z)<0\}$. \\[.5mm] 
$(iii)$ $\si(A)$ is one of the following: \\[.5mm] 
$(iii_1)$ $\ol{\C_+}=\{z\in\C\st \Im(z)\ge0\}$. $($if $\dim \, (\ker(A^*+iI))>0$, 
$\dim \, (\ker(A^*-iI))=0$$)$. \\[.5mm]  
$(iii_2)$ $\ol{\C_-}=\{z\in\C\st \Im(z)\le0\}$. $($if $\dim \, (\ker(A^*-iI))>0$, 
$\dim \, (\ker(A^*+iI))=0$$)$. \\[.5mm]  
$(iii_3)$ $\C$. $($if $\dim \, (\ker(A^*\pm iI))>0$$)$. \\[.5mm] 
$(iii_4)$ a subset of $\R$ $($if $\dim\ker(A^*\pm iI))=0$$)$.  \\[.5mm] 
$(iv)$ $A = A^*$ if and only if $(iii_4)$ holds $($i.e., if and only if 
$\si(A)\subseteq\R$.$)$. \\[.5mm] 
$(v)$ $A = A^*$ if and only if $\dim \, (\ker(A^*-zI))=0$ for all $z\in\C_+\cup\C_-$. 
\end{theorem}
%%%%%%%%%%

%%%%%%%%%%
\begin{corollary} \lb{c2.3}
Suppose $A$ is symmetric and closed in $\cH$ and bounded from below, that is, there exists $c\in\R$ such that
\begin{equation}
(f,A f)_{\cH} \ge c\|f\|_{\cH}^2, \quad f\in\dom(A).    \lb{2.1} 
\end{equation}
Then $\dim \, (\ker(A^*-zI))$ is independent of $z\in\C\bs[c,\infty)$.
\end{corollary}
%%%%%%%%%%

%%%%%%%%%%
\begin{definition}[Deficiency indices] \lb{d2.4}
Suppose $A$ is symmetric in $\cH$. Define the \textit{deficiency subspaces of $A$} $($resp., $\ol A$$)$ by
\begin{align}
\begin{split} 
\cH_\pm(A) &= \ker(A^*\mp iI) = (\ran \, (A \pm iI))^\perp    \\ 
&= \cH_\pm(\ol{A}) \;\text{ $($since $A \subseteq \ol A\subseteq A^* = (\ol A)^*$$)$}.    \lb{2.2} 
\end{split} 
\end{align}
Then,
\begin{align}
\begin{split} 
n_\pm(A)&=\dim \, (\cH_\pm(A))=\dim \, (\ker(A^* \mp iI))    \\
&=\dim\big((\ran \, (A \pm iI)^\perp\big)
=n_\pm\big(\ol A\big)   \lb{2.3} 
\end{split} 
\end{align}
are called the \textit{deficiency indices of $A$} $($resp., $\ol A$$)$ and one introduces 
\begin{equation}
\Def(A) = \big[n_+(A) + n_-(A) \big)]/2.   \lb{def} 
\end{equation} 
\end{definition}
%%%%%%%%%%

One recalls that for a symmetric and closed operator $A$ one has von Neumann's first formula
\begin{equation}
\dom(A^*) = \dom(A) \dotplus \cH_+(A) \dotplus \cH_-(A).      \lb{2.3a} 
\end{equation}
In particular, 
\begin{align}
\begin{split} 
&\dom(A^*)=\big\{f + f_+ + f_- \st f \in\dom(A), \; f_\pm \in \cH_\pm(A)\big\},     \\
& \, A^*(f + f_+ + f_-) = A f + i f_+ - i f_-, \quad f \in \dom(A), \; 
f_\pm \in \cH_\pm(A).      \lb{2.3b} 
\end{split} 
\end{align}

If $A$ is symmetric, $A \subseteq A^*$, and $B \supseteq A$ is a symmetric extension of $A$, then 
\begin{equation}
A \subseteq B \subseteq B^*\subseteq A^*.       \lb{2.4}
\end{equation}

%%%%%%%%%%
\begin{theorem}[von Neumann] \lb{t2.5}
Suppose $A$ is symmetric and closed in $\cH$. Then the closed symmetric extensions of $A$ are in 
one-to-one correspondence with the set of partial isometries $($in the usual inner product of $\cH$$)$ 
of $\cH_+(A)$ into $\cH_-(A)$. Let $U:I(U)\to\cH_-(A)$ with $I(U)\subseteq\cH_+(A)$ be such an 
isometry with initial space $I(U)$, with $I(U)$ closed in $\cH$. Then the corresponding closed and symmetric 
extension $A_U$ of $A$ is given by
\begin{align}
\begin{split} 
&\dom(A_U)=\big\{f + f_++U f_+ \st f \in \dom(A), \; f_+\in I(U)\big\},     \\
& \, A_U(f + f_++U f_+) = A f + i f_+ - iU f_+, \quad f \in \dom(A), \; f_+ \in \cH_+(A).      \lb{2.5} 
\end{split} 
\end{align}
If $\dim \, (I(U))<\infty$, then
\begin{align}
n_\pm(A_U)=n_\pm(A) - \dim \, (I(U)).     \lb{2.6} 
\end{align}
\end{theorem}
%%%%%%%%%%

%%%%%%%%%%
\begin{corollary} \lb{c2.6}
Let $A$ be closed and symmetric in $\cH$. Then the following assertions hold: \\[.5mm] 
$(i)$ $A$ is self-adjoint if and only if $n_+(A) = n_-(A)=0$. \\[.5mm] 
$(ii)$ $A$ has self-adjoint extensions if and only if $n_+(A) = n_-(A)$. In such a scenario, there is a one-to-one correspondence between self-adjoint extensions of $A$ and unitary maps from $\cH_+(A)$ 
onto $\cH_-(A)$. Therefore, if $n_\pm(A)=n<\infty$, the family of self-adjoint extensions 
of $A$ is a real $n^2$-parameter family. \\[.5mm] 
$(iii)$ If either $n_+(A)=0$, or $n_-(A)=0$, then $A$ has no proper symmetric extensions. $($Such 
operators are called maximally symmetric.$)$ 
\end{corollary}
%%%%%%%%%%

By Corollary \ref{c2.3}, any symmetric operator $A$ in $\cH$ bounded from below has equal deficiency indices and hence self-adjoint extensions.

%%%%%%%%%%
\begin{definition} \lb{d2.7} 
An anti linear map $\cC:\cH\to\cH$ $($i.e., $\cC(c_1 f_1+ c_2 f_2) = {\ol c_1}\cC f_1 + 
{\ol c_2} \cC f_2$, $c_j \in \C$, $f_j \in \cH$, $j=1,2$$)$ is called a \textit{conjugation} if it is norm preserving $($i.e., $\|\cC f\|_{\cH} = \|f\|_{\cH}$, $f\in\cH$$)$ and $\cC^2=I_\cH$.
\end{definition}
%%%%%%%%%%

The prime example of a conjugation map is of course complex conjugation in $L^2(\Om)$, that is, 
$\cC:\begin{cases} L^2(\Om) \to L^2(\Om), \\ f \mapsto \cC f=\ol f. \end{cases}$

%%%%%%%%%%
\begin{theorem} \lb{t2.8}
Suppose $A$ is symmetric in $\cH$ and $\cC$ is a conjugation in $\cH$ with $\cC\dom(A) \subseteq \dom(A)$ and $A \cC = \cC A$. Then $n_+(A) = n_-(A)$ and hence $A$ has self-adjoint extensions.
\end{theorem}
%%%%%%%%%%

%%%%%%%%%%%%%%%%%%%%%%%%%
%%%%%%%%%%%%%%%%%%%%%%%%%
\section{On Relative Form and Operator Boundedness}  \lb{s3}
%%%%%%%%%%%%%%%%%%%%%%%%%
%%%%%%%%%%%%%%%%%%%%%%%%%

To set the stage, we briefly recall the notion of relatively bounded and relatively form bounded perturbations of a closed operator $A$ in some complex separable Hilbert space $\cH$: 

%%%%%%%%%%%%%%%
\begin{definition} \lb{d3.1}
$(i)$ Suppose that $A$ is a closed operator in $\cH$. A closable operator $B$ in $\cH$ is called {\it relatively bounded} {\it with respect to $A$} $($in short, $B$ is called {\it $A$-bounded\,}$)$, if
\begin{align}
& \dom(A) \subseteq  \dom(B),     \no \\
& \text{and for some constants $a, b \in [0,\infty)$, }    \lb{3.3A} \\ 
& \|Bf\|_{\cH} \le a \|Af\|_{\cH} + b \|f\|_{\cH}, \quad f\in\dom(A).    \no 
\end{align}
$(ii)$ Assume, in addition, that $A$ is self-adjoint in $\cH$ and bounded from below, that is,
$A \ge c I_{\cH}$ for some $c\in\R$. Then a densely defined and closed operator $B$ in $\cH$ is called {\it relatively form bounded} {\it with respect to $A$} $($in short, $B$ is called {\it $A$-form bounded}\,$)$, if
\begin{equation}
\dom\big(|A|^{1/2}\big) \subseteq \dom\big(|B|^{1/2}\big).      \lb{3.4A}
\end{equation}
\end{definition}
%%%%%%%%%%%%%%%

\medskip

In particular, $B$ is $A$-form bounded if and only if $|B|$ is. Using the polar decomposition of $B$ (i.e., $B=U|B|$), one observes that $B$ is $A$-bounded if and only if $|B|$ is $A$-bounded.

We also recall that in connection with relative boundedness, the first condition in \eqref{3.3A}, 
$ \dom(A) \subseteq  \dom(B)$, already implies the second condition, viz., there exist numbers
$a, b \in [0,\infty)$ such that $\|Bf\|_{\cH} \le a \|Af\|_{\cH} + b \|f\|_{\cH}$, $f\in\dom(A)$, or equivalently, 
\begin{align}
\begin{split}
& \text{there exist numbers $\wti a, \wti b \in [0,\infty)$ such that} \\
& \|Bf\|^2_{\cH} \le {\wti a}^2 \|Af\|^2_{\cH} + {\wti b}^2 \|f\|^2_{\cH}, \quad f\in\dom(A).      \lb{3.6A}
\end{split}
\end{align}

We also note that if $A$ is self-adjoint and bounded from below, the number $\alpha$ defined by
\begin{equation}
\alpha = \lim_{\mu\uparrow \infty} \big\|B(A+\mu I_{\cH})^{-1}\big\|_{\cB(\cH)}
= \lim_{\mu\uparrow \infty} \big\||B|(A+\mu I_{\cH})^{-1}\big\|_{\cB(\cH)}     \lb{3.7A}
\end{equation}
equals the greatest lower bound (i.e., the infimum) of the possible values for $a$ in \eqref{3.3A} 
(resp., for $\wti a$ in \eqref{3.6A}). This number $\alpha$ is called the $A$-bound of $B$. 
Similarly, we call
\begin{equation}
\beta = 
\lim_{\mu\uparrow \infty} \big\||B|^{1/2}\big(|A|^{1/2}+\mu I_{\cH}\big)^{-1}\big\|_{\cB(\cH)}     \lb{3.7aA}
\end{equation}
the $A$-form bound of $B$ (resp., $|B|$). If $\alpha =0$ in \eqref{3.7A} (resp., $\beta =0$ in \eqref{3.7aA}) then $B$ is called {\it infinitesimally bounded} (resp., {\it infinitesimally form bounded}\,) with respect to $A$.   

Our first result is an abstract version of Morgan \cite[Theorem~2.1]{Mo79} (see also 
\cite[Proposition~3.3]{CFKS87}, \cite{Is61}, \cite[Sect.\ 4]{Ki81}). Throughout this 
section, infinite sums are understood in the weak operator topology and $J \subseteq \N$ 
denotes an index set.

%%%%%%%%%
\begin{theorem} \lb{t3.1}
Suppose that $T$, $W$ are self-adjoint operators in $\cH$ such that  
\begin{equation} 
\dom\big(|T|^{1/2}\big)\subseteq\dom\big(|W|^{1/2}\big),  
\end{equation} 
and let $c, d \in (0,\infty)$, $e \in [0,\infty)$.
Moreover, suppose $\Phi_j \in \cB(\cH)$, $j \in J$, leave $\dom\big(|T|^{1/2}\big)$ invariant, 
that is, 
\begin{equation} 
\Phi_j \dom\big(|T|^{1/2}\big) \subseteq \dom\big(|T|^{1/2}\big), \quad j \in J, 
\end{equation} 
and satisfy the following conditions $(i)$--$(iii)$: \\[.5mm]  
$(i)$ $\sum_{j  \in J} \Phi_j^* \Phi_j \leq I_{\cH}$. \\[.5mm] 
$(ii)$ $\sum_{j  \in J}  \Phi_j^* |W| \Phi_j \ge c^{-1} |W|$ on $\dom\big(|T|^{1/2}\big)$.\\[.5mm] 
$(iii)$ $\sum_{j  \in J}  \| |T|^{1/2} \Phi_j f\|_{\cH}^2 \le d\| |T|^{1/2} f\|_{\cH}^2 
+ e\|f\|_{\cH}^2$, \, $f \in \dom\big(|T|^{1/2}\big)$. \\[.5mm] 
Then,
\begin{equation}
\big\| |W|^{1/2} \Phi_j f \big\|_{\cH}^2 \le a \big\||T|^{1/2} \Phi_j f \big\|_{\cH}^2 
+ b \|\Phi_j f\|_{\cH}^2, \quad f \in\dom(|T|^{1/2}), \; j\in J,   \lb{3.9} 
\end{equation} 
implies
\begin{equation}
\big\| |W|^{1/2} f \big\|_{\cH}^2 \le a\,c\,d \big\||T|^{1/2} f \big\|_{\cH}^2 
+ [a\,c\,e + b\, c] \|f\|_{\cH}^2, \quad f \in\dom(|T|^{1/2}). 
\end{equation} 
\end{theorem}
%%%%%%%%%
\begin{proof}
For $f \in \dom\big(|T|^{1/2}\big)$ one computes
\begin{align} 
\big\| |W|^{1/2} f \big\|_{\cH}^2 &\le c \sum_{j  \in J}  \big\| |W|^{1/2} \Phi_j f \big\|_{\cH}^2
\quad \text{(by $(ii)$)}   \no \\ 
&\le c \sum_{j  \in J}  \left[ a \big\||T|^{1/2} \Phi_j f \big\|_{\cH}^2 + b \|\Phi_j f\|_{\cH}^2 \right]
\quad \text{(by \eqref{3.9})}   \no \\ 
&\le a\, c\, d   \big\|  |T|^{1/2} f \big\|_{\cH}^2 + (a\,c\,e + b\, c) \|f\|_{\cH}^2 
\quad \text{(by $(i)$ and $(iii)$),} 
\end{align}
finishing the proof.
\end{proof}
%%%%%%%%%

%%%%%%%%%
\begin{remark}
The following condition $(iii')$, viz., for some $\tilde e \in (0,\infty)$, \\[.5mm]  
$(iii')$ $\sum_{j \in J}  \big\|[|T|^{1/2}, \Phi_j] f \big\|_{\cH}^2 \le \tilde{e} \|f\|_{\cH}^2$, \, 
$f \in \dom\big(|T|^{1/2}\big)$, \\[.5mm]
together with condition $(i)$, implies condition $(iii)$ with $d=1+\e$ ($\e>0$ arbitrarily small) 
and $e=\frac{(1+\e)\tilde e}{\e}$. (Here $[\, \cdot \,, \, \cdot \,]$ denotes the commutator symbol.) 
To see this, one uses the triangle inequality, 
\begin{equation}
\big\|[|T|^{1/2}, \Phi_j] f \big\|_{\cH} \geq \big\||T|^{1/2} \Phi_j f \big\|_{\cH}
-\big\|\Phi_j|T|^{1/2} f \big\|_{\cH}
\end{equation}
as well as the observation that, for any real numbers  $a$ and $b$, 
\begin{equation}
\frac{1+\e}{\e}(a-b)^2\geq a^2-(1+\e)b^2 \, \text{ if and only if } \, 
\frac{1}{1+\e}a^2-2ab+(1+\e)b^2\geq 0.
\end{equation}

While condition $(iii')$ might look slightly more natural in
our context, the formulation in condition $(iii)$ is advantageous in cases where 
$T$ has an explicit factorization as $T=A^* A$, but a straightforward formula for $|T|^{1/2}$ is 
not available, since in such cases one can use the fact that 
\begin{equation} 
\big\|  |T|^{1/2} f \big\|_{\cH}^2=\big\| A f \big\|_{\cH}^2, \quad 
f\in\dom\big(|T|^{1/2}\big) = \dom(A). 
\end{equation} 

Finally, we note that condition $(iii')$ itself is implied by the fully localized 
condition:\\[.5mm]
$(iii'')$ $\big\|[|T|^{1/2}, \Phi_j] f \big\|_{\cH}^2 \le \tilde{e} \|\Phi_j f\|_{\cH}^2$, \,
$f \in \dom\big(|T|^{1/2}\big)$, $j\in J$. \hfill $\diamond$ 
\end{remark}
%%%%%%%%%

In particular, consider the concrete case of 
\begin{equation}
T=-\Delta, \quad \dom(T) = H^2(\R^n),   \lb{3.3}
\end{equation}
in $L^2(\R^n)$, and assume that $W$, the operator of multiplication with a real-valued function $W$ 
(with a slight abuse of notation), is relatively 
form bounded with respect to $T=-\Delta$ (for sufficient conditions on $W$, see, e.g., \cite[Theorems~10.17\,(b), 10.18]{We80} with $r=1$).
Let $\{\phi_j\}_{j\in J}$, $J\subseteq \N$, be a family of smooth, real-valued functions defined 
on $\R^n$ in such a manner that for each $x \in \R^n$, there exists an open neighborhood $U_x \subset \R^n$ of $x$ such that there exist only finitely many indices $k \in J$ with 
$\supp \, (\phi_k) \cap U_x \neq \emptyset$ and  $\phi_k|_{U_x} \neq 0$, as well as 
\begin{equation}
\sum_{j\in J} \phi_j(x)^2=1, \quad x \in \R^n    \lb{3.15}
\end{equation}
(the sum over $j \in J$ in \eqref{3.15} being finite). Finally, let $\Phi_j$ be the operator of 
multiplication by the function $\phi_j$, 
$j\in J$. Then one notes that for these choices, hypothesis $(i)$ holds with equality, and 
hypothesis $(ii)$ with $c=1$ follows from $(i)$. 
Moreover, item $(iii)$ holds with $d=1$ as long as 
\begin{equation}
e = \bigg\| \sum_{j \in J}  |\nabla \phi_j(\cdot)|^2 \bigg\|_{L^{\infty}(\R^n)} < \infty.
\end{equation}
To verify this, one observes that 
$\| |T|^{1/2} \phi f\|_{L^2(\R^n)}^2 = \int_{\R^n} d^n x \, |\nabla (\phi(x) f(x))|^2$ 
and that the cross terms vanish since $\sum_{j \in J}  \phi_j(x) (\nabla \phi_j)(x) =0$, 
$x \in \R^n$, by condition $(i)$. (We note again that the latter sum over $j \in J$ contains only 
finitely many terms in every bounded neighborhood of $x \in \R^n$.) 
This is precisely \cite[Theorem~2.1]{Mo79}. 
Strongly singular potentials that are covered by Theorem \ref{t3.1} are, for instance, of the 
following form: Let $J\subseteq \N$ be an index set, and $\{x_j\}_{j\in J}\subset\R^n$, 
$n \in \N$, $n \geq 3$, be a set of points such that 
\begin{equation}
\inf_{\substack{j, j' \in J \\ j \neq j'}} |x_j - x_{j'}| > 0.    
\end{equation} 
In addition, let $\gamma_j \in \R$, $j \in J$, $\gamma, \delta \in (0, \infty)$ with 
\begin{equation}
|\gamma_j| \leq \gamma < (n -2)^2/4, \; j \in J, 
\end{equation}
and 
\begin{equation}
W(x) = \sum_{j \in J} \gamma_j \f{e^{- \delta |x - x_j|}}{|x - x_j|^2}, 
\quad x \in \R^n \backslash \{x_j\}_{j \in J}.    \lb{3.20} 
\end{equation}
Then an application of Hardy's inequality in $\R^n$, $n \geq 3$, shows that $W$ is form 
bounded with respect to $T$ in \eqref{3.3} with form bound strictly less than one 
(cf.\ \cite[p.~28--29]{CFKS87}).

Similarly, one obtains the following operator perturbation analog of the form perturbation 
result in Theorem \ref{t3.1}.

%%%%%%%%%%
\begin{theorem} \lb{t3.2}
Suppose that $T$, $W$ are symmetric in $\cH$ such that 
\begin{equation} 
\dom(T)\subseteq\dom(W),   
\end{equation} 
and let $c > 0$, $d > 0$, $e \geq 0$. 
Moreover, suppose $\Phi_j \in \cB(\cH)$, $j \in J$, leave $\dom(T)$ invariant, that is, 
\begin{equation} 
\Phi_j \dom(T) \subseteq \dom(T), \quad j \in J, 
\end{equation} 
and satisfy the following conditions $(i)$--$(iii)$: \\[.5mm]  
$(i)$ $\sum_{j \in J}  \Phi_j^* \Phi_j \le I_{\cH}$. \\[.5mm] 
$(ii)$ $\sum_{j \in J}  \Phi_j^* W^2 \Phi_j \ge c^{-1} W^2$ on $\dom(T)$.\\[1mm] 
$(iii)$ $\sum_{j \in J}  \| T \Phi_j f\|_{\cH}^2 \le d\| T f\|_{\cH}^2 + e\|f\|_{\cH}^2$, 
\, $f \in \dom(T)$.\\[.5mm]
$($Here infinite sums are understood in the weak operator topology.$)$ Then,
\begin{equation}
\| W \Phi_j f\|_{\cH}^2 \le a \|T \Phi_j f\|_{\cH}^2 + b \|\Phi_j f\|_{\cH}^2, 
\quad f \in\dom(T), \; j\in J,     \lb{3.23} 
\end{equation} 
implies
\begin{equation}
\| W f\|_{\cH}^2 \le a\,c\,d \|T f\|_{\cH}^2 + [a\,c\,e + b\, c] \|f\|_{\cH}^2, \quad f \in\dom(T). 
\end{equation}
\end{theorem}
%%%%%%%%%%
\begin{proof}
For $f \in \dom(T)$ one computes
\begin{align} 
\| W f\|_{\cH}^2 &\le c \sum_{j \in J}  \| W \Phi_j f\|_{\cH}^2 
\quad \text{(by $(ii)$)}  \\ 
&\le c \sum_{j \in J}  \left[ a \|T \Phi_j f\|_{\cH}^2 + b \|\Phi_j f\|_{\cH}^2 \right] 
\quad \text{(by \eqref{3.23})}   \\ 
&\le a\, c\, d  \| T f\|_{\cH}^2 + (a\,c\,e + b\, c) \|f\|_{\cH}^2  
\quad \text{(by $(i)$ and $(iii)$),}  
\end{align}
concluding the proof.
\end{proof}
%%%%%%%%%%

Again one notes that item $(iii)$ holds with $d=1+\e$, $\e > 0$, and $e>0$ if 
\begin{equation} 
\sum_{j \in J}  \|[T, \Phi_j] f\|^2_{\cH} \le \frac{\e^2}{4+2\e} \|T f\|^2_{\cH} 
+ \frac{\e e}{2+\e} \|f\|^2_{\cH}, \quad f \in \dom(T).    \lb{3.28} 
\end{equation} 

As an immediate consequence of  \cite[Corollary~2]{BF77} one obtains the following result.

%%%%%%%%%%
\begin{corollary} \lb{c3.4} 
In addition to the assumptions of Theorem \ref{t3.2}, suppose that $T$ is closed, and that 
$a\,c\,d <1$. Then
\begin{equation} 
n_\pm(T+W) = n_\pm(T).
\end{equation} 
\end{corollary}
%%%%%%%%%%

In particular, in the case of $T=-\Delta$ in $L^2(\R^n)$ as in \eqref{3.3} and 
$W$ the operator of multiplication with a real-valued function $W$ (again, abusing 
notation) bounded with respect to $T=-\Delta$ (for sufficient conditions on $W$, see, 
e.g., \cite[Theorems~10.17\,(b), 10.18]{We80} with $r=2$), and $\Phi_j$ the operator of 
multiplication by a smooth function $\phi_j$, item $(ii)$ with $c=1$ follows from $(i)$ in case 
one has equality on the support of $W$ in item $(i)$. Moreover,
\begin{equation}
[ T, \Phi_j ] f = (-\Delta \phi_j) f - 2 (\nabla \phi_j) \cdot (\nabla f), \quad f \in \dom(T),
\end{equation}
implying
\begin{align}
\begin{split} 
& \sum_{j \in J}  \|[ T, \Phi_j ] f\|_{L^2(\R^n)}^2 \le \alpha \|f\|_{L^2(\R^n)}^2 
+ \beta \|\nabla f\|_{L^2(\R^n)^n}^2,   \\
& \, \alpha = 2 \bigg\|\sum_{j \in J}  (\Delta \phi_j)^2 \bigg\|_{L^{\infty}(\R^n)}, 
\quad \beta = 4 \bigg\|\sum_{j \in J}  (\nabla \phi_j)^2 \bigg\|_{L^{\infty}(\R^n)}.
\end{split}
\end{align}
Finally, the elementary inequality, 
\begin{equation}
\|\nabla f\|_{L^2(\R^n)^n} \le \e \|\Delta f\|_{L^2(\R^n)} + (2\e)^{-1} \|f\|_{L^2(\R^n)}, 
\quad f \in H^2(\R^n), \; \varepsilon > 0, 
\end{equation}
shows that \eqref{3.28} holds if $\alpha,\beta \in [0,\infty)$.

Next, introducing the space of uniformly locally $L^p$-functions by
\begin{equation}
L^p_{\rm loc \, unif}(\R^n) := \big\{f \in L^p_{\loc}(\R^n) \, \big| \, 
\sup_{x \in \R^n} \|f \chi_{B_n(x;1)}\|_{L^p(\R^n)} < \infty\big\}, \quad p \in [1,\infty), 
\end{equation} 
one can derive a quick proof of \cite[Theorem~XIII.96]{RS78}:  

%%%%%%%%%
\begin{corollary} \lb{c3.5}
Let $V \in L^p_{\rm loc \, unif}(\R^n)$ be real-valued with $p=2$ for $n=1,2,3$ and $p > n/2$ for $n \geq 4$. Then $V$ is infinitesimally bounded $($and hence infinitesimally form bounded\,$)$ with respect to $H_0= - \Delta$, $\dom(H_0) = H^2(\R^n)$, in $L^2(\R^n)$.
\end{corollary}
%%%%%%%%%
\begin{proof}
Let $\phi$ be a nonnegative smooth function which equals $1$ in $B_n(0;1/2)$ and vanishes outside $B_n(0;1)$. Let $x_j$, $j \in J$, be the points of a periodic lattice such that 
$\sum_{j \in J} \phi(x-x_j)^2 \ge 1/2$, $x\in\R^n$. Set $\phi_j(x) = \phi(x-x_j) \big[\sum_{j' \in J} \phi(x-x_{j'})^2\big]^{-1/2}$, $x \in \R^n$, $j \in J$, such that 
$\sum_{j\in J} \phi_j(x)^2 =1$, $x \in \R^n$. Then items $(i)$--$(iii)$ hold with $c=1$, $d=1+\e$ as pointed out above. Moreover, by \cite[Theorem~X.15 and X.20]{RS75}, for any $\varepsilon > 0$, one can find a corresponding $b(\varepsilon) > 0$, such that
\begin{equation}
\| V \phi_j f\|_{L^2(\R^n)}^2 \le \varepsilon \|H_0 \phi_j f\|_{L^2(\R^n)}^2 
+ b(\varepsilon) \|\phi_j f\|_{L^2(\R^n)}^2, \quad f \in H^2(\R^n), \; j \in J, 
\end{equation}
proving Corollary \ref{c3.5}.
\end{proof}
%%%%%%%%%%

For a form version of Corollary \ref{c3.5} we refer to the comments surrounding equation (2.10) in 
Morgan \cite[p.~112]{Mo79}. 

It is well-known that for $n=1$, $V \in L^2_{\rm loc \, unif}(\R^n)$ is equivalent to $V$ being 
relatively bounded with respect to $H_0$, which in turn is, in fact, equivalent to 
$V$ being infinitesimally bounded with respect to $H_0$, see, \cite[Theorem~2.7.1]{Sc81}.  
(See also \cite{GW14} for more results and literature on the one-dimensional case.)  
For necessary and sufficient conditions on form boundedness (resp., infinitesimal form 
boundedness) of $V$ relative to $H_0$ in the multi-dimensional case we refer to \cite[Theorem~4.2]{MV02} (resp., \cite[Theorem~III]{MV05}), see also \cite[Sects.~2.5, 11.4]{MS09}.

%%%%%%%%%%%%%%%%%%%%%%%%
%%%%%%%%%%%%%%%%%%%%
\section{Decoupling of Deficiency Indices. An Abstract Approach} \label{s4}
%%%%%%%%%%%%%%%%%%%%%%%%
%%%%%%%%%%%%%%%%%%%

Next we turn to a general scheme of determining deficiency indices 
particularly suited for dealing with Schr\"odinger-type 
operators with potentials that exhibit strong singularities at (possibly, countably many) 
uniformly separated points (or compact 
sets of $n$-dimensional Lebesgue measure zero) to be explored in Section \ref{s5}.

%%%%%%%%
\begin{hypothesis} \lb{h2.9}
Let $J \subseteq \N$ be an index set
and let $T$, $T_j$, $j\in J$ be closed symmetric operators in $\cH$. 
Suppose there exist $\Phi_j \in \cB(\cH)$ and $\wti\Phi_j \in \cB(\cH)$, $j \in J$, such that 
\begin{equation}
\Phi = \sum_{j \in J} \Phi_j \in \cB(\cH), \lb{2.8a} 
\end{equation}
with convergence in the strong operator topology, and 
\begin{align} 
& \wti \Phi_j \Phi_j = \Phi_j, \; j \in J, \quad \wti \Phi_j \Phi_k = 0, \; j, k \in J, \, j \neq k,     \lb{2.9} \\ 
& g \in \dom(T_j^*) \, \text{ implies } \, \Phi_j g \in \dom(T^*), \quad j \in J,    \lb{2.10} \\
& g \in \dom(T^*) \, \text{ implies } \, \Phi_j g \in \dom(T_j^*), \quad j \in J,    \lb{2.11} \\
& g \in \dom(T_j) \, \text{ implies } \, \wti \Phi_j g \in \dom(T), \quad j \in J,    \lb{2.11a} \\ 
& g \in \dom(T) \, \text{ implies } \, \wti \Phi_j g \in \dom(T_j), \quad j \in J,    \lb{2.12} \\
& g \in \dom(T_j^*) \, \text{ implies } \, (I_{\cH} - \Phi_j) g \in \dom(T_j), \quad j \in J,    \lb{2.13} \\
& g \in \dom(T^*) \, \text{ implies } \, \bigg(I_{\cH} - \sum_{j \in J} \Phi_j\bigg) g \in \dom(T),    \lb{2.14} \\
& g \in \dom(T^*) \, \text{ implies } \,  
\slim_{N \to \infty} T^* \bigg(\sum_{\substack{j \in J \\ |j| \leq N}} \Phi_j g\bigg) = 
T^* \bigg(\sum_{j \in J} \Phi_j g\bigg).    \lb{2.14d}
\end{align} 
$($Condition \eqref{2.14d} is redundant if $\# (J) < \infty$.$)$
\end{hypothesis}
%%%%%%%%

Next, we recall the notion of linear independence with respect to a linear subspace of $\cH$: Let 
$\cD \subseteq \cH$ be a linear subspace of $\cH$. Then the vectors $f_k \in \cH$, $1 \leq k \leq N$, $N \in \N$, 
are called {\it linearly independent} $({\rm mod}\, \cD)$, if 
\begin{align}
\begin{split} 
&\sum_{k = 1}^N c_k f_k \in \cD \, \text{ for some coefficients } \, c_k \in \C, \, 1 \leq k \leq N, \\
& \quad \text{implies } \, c_k = 0, \, 1 \leq k \leq N.     \lb{2.17}
\end{split} 
\end{align}
In addition, with $\cD$ and $\cE$ linear subspaces of $\cH$ with $\cD \subseteq \cE$, the quotient subspace 
$\cE / \cD$ consists of equivalence classes $[f]$ such that $g \in [f]$ if and only if $(f - g) \in \cD$, in 
particular, $f = g \; ({\rm mod}\, \cD)$ is equivalent to $(f -g) \in \cD$. Moreover, the dimension of 
$\cE \; ({\rm mod} \, \cD)$, denoted by $\dim_{\cD}(\cE)$, equals 
$n \in \N$, if there are $n$, but not more than $n$, linearly independent vectors 
in $\cE$, such that no linear combination (except, the trivial one) belongs to $\cD$. 
If no such finite $n\in\N$ exists, one defines $\dim_{\cD}(\cE) = \infty$. 
Consequently, if $A$ is symmetric and closed, then \eqref{2.3a} implies
\begin{equation}\label{defectdom}
\dim_{\dom(A)}(\dom(A^*)) = n_+(A) + n_-(A) =2 \Def(A). 
\end{equation}
 
The following result computes the defect of $T$ in terms of those of $T_j$, $j \in J$.
  
%%%%%%%%%%
\begin{theorem} \lb{t2.10}
Assume Hypothesis \ref{h2.9}. Then 
\begin{equation}
\Def\,(T) = \sum_{j \in J} \Def(T_j),     \lb{2.18} 
\end{equation} 
including the possibility that one, and hence both sides of \eqref{2.18} equal $\infty$. 
\end{theorem}
%%%%%%%%%%
\begin{proof}
We start with the special case where 
\begin{equation} 
\sum_{j \in J} \Def(T_j) < \infty,      \lb{2.18a} 
\end{equation} 
that is, for at most finitely many $k \in J$, $0 < \Def(T_k) = N_k$ for some $N_k \in\N$. In 
this context we abbreviate  
\begin{equation}
K = \{k \in J \,|\, \Def(T_k) > 0\} \subseteq J,\quad K \, \text{ finite}. 
\end{equation}
For each $k\in K$, let
\begin{equation}\lb{2.19}
f_{k,\ell} \in \dom(T_k^*),\quad  1 \leq \ell \leq 2 \Def(T_k),
\end{equation} 
be a maximal set of vectors in $\dom(T_k^*)$, linearly independent 
$({\rm mod}\, \dom(T_k))$.   
Then \eqref{2.10} yields 
\begin{equation}
\Phi_k f_{k,\ell} \in \dom(T^*), \; 1 \leq \ell \leq 2 \Def(T_k), \, k \in K.     \lb{2.20}   
\end{equation}
Next, let $\beta_{k, \ell} \in \C$, $1 \leq \ell \leq 2 \Def(T_k)$, 
$k \in K$, be constants such that 
\begin{equation}
\sum_{k \in K} \sum_{\ell = 1}^{2 \Def(T_k)} \beta_{k, \ell} \Phi_k f_{k, \ell} \in \dom(T).     \lb{2.21}
\end{equation}
(Here and in what follows, note that \eqref{2.18a} and the fact that $K$ is finite imply that every sum as
above has finitely many terms.) 
Applying \eqref{2.12} and subsequently \eqref{2.9} then yields that for every $k\in K$,
\begin{equation}
\dom(T_k)\ni \wti \Phi_{k} \bigg(\sum_{k' \in K} \sum_{\ell = 1}^{2 \Def(T_k')} \beta_{k', \ell} \Phi_{k'} f_{k', \ell}\bigg) 
= \sum_{\ell = 1}^{2 \Def(T_{k})} \beta_{k, \ell} \Phi_{k} f_{k, \ell}.     \lb{2.22}
\end{equation}
On the other hand, from  \eqref{2.13} and the fact that $f_{k,\ell} \in \dom(T_k^*)$, one finds that
for every $k\in K$, 
\begin{equation}
(I_{\cH} - \Phi_k) \sum_{\ell = 1}^{2 \Def(T_k)} \beta_{k, \ell} f_{k, \ell} \in \dom(T_k). 
\lb{2.23}
\end{equation}

Combining \eqref{2.22} and \eqref{2.23} one concludes
\begin{equation}
\sum_{\ell = 1}^{2 \Def(T_k)} \beta_{k, \ell} f_{k, \ell} \in \dom(T_k), \; k \in K,    \lb{2.24}
\end{equation}
and hence 
\begin{equation}
\beta_{k, \ell} = 0, \; 1 \leq \ell \leq 2 \Def(T_k), \; k \in K,     \lb{2.25}
\end{equation}
since for every $k \in K$, $f_{k,\ell} \in \dom(T_k^*)$, $1 \leq \ell \leq 2 \Def(T_k)$, were chosen linearly independent $({\rm mod}\, \dom(T_k))$. Consequently,
\begin{equation}
\Phi_k f_{k,\ell} \in \dom(T^*), \quad 1 \leq \ell \leq 2 \Def(T_k), \; k \in K,     \lb{2.26}   
\end{equation}
are linearly independent $({\rm mod}\, \dom(T))$, implying
\begin{equation}
2 \Def(T) = \dim_{\dom(T)} (\dom(T^*)) \geq 2 \sum_{k \in K} \Def(T_k)  
= 2 \sum_{j \in J} \Def(T_j).     \lb{2.27}
\end{equation}

Now suppose (by contradiction) that equality does not hold in \eqref{2.26}, that is,   
\begin{equation}
\Def(T) > \sum_{k \in K} \Def(T_k) = \sum_{j \in J} \Def(T_j).      \lb{2.36} 
\end{equation}
Since all the $\Phi_k f_{k,l}$, $1 \leq \ell \leq 2 \Def(T_k)$, 
$k \in K$, are linearly independent in $\dom(T^*)$ $({\rm mod}\, \dom(T))$, assumption \eqref{2.36} implies that there exists $f \in \dom(T^*)$ such that 
\begin{equation}
f, \Phi_k f_{k,\ell}, \quad 1 \leq \ell \leq 2  \Def(T_k), \; k \in K,    \lb{2.37}
\end{equation}
are still linearly independent $({\rm mod}\, \dom(T))$. 
One notes once again that by \eqref{2.13}, $(I_{\cH} - \Phi_k) f_{k,\ell} \in \dom(T_k)$, 
$1 \leq \ell \leq 2 \Def(T_k)$, $k \in K$, and hence one can write, 
\begin{equation}\label{2.37a}
f_{k,\ell} = \Phi_k f_{k,\ell} + g_{k,\ell}, \quad g_{k,\ell} \in \dom(T_k), \; 
1 \leq \ell \leq 2 \Def(T_k), \; k \in K. 
\end{equation} 
Applying \eqref{2.11}, one 
concludes that $\Phi_j f \in \dom(T_j^*)$, $j \in J$. 

If $j\in K$, then there exist coefficients 
$c_{k,\ell} \in \C$, $1 \leq \ell \leq 2 \Def(T_k)$, $k \in K$, and an element 
$\wti g_j \in\dom(T_j)$, such that 
\begin{equation}
\Phi_j f = \sum_{\ell=1}^{2 \Def(T_j)} c_{j,\ell} f_{j,\ell} + \wti g_j
=\sum_{\ell=1}^{2 \Def(T_j)} c_{j,\ell} \Phi_j f_{j,\ell} + g_j,
\end{equation}
where in the second identity we have used \eqref{2.37a}, and we have set 
\begin{equation}
g_j= \Bigg(\sum_{\ell=1}^{2\Def(T_j)}c_{j,\ell}g_{j,\ell}+\wti g_j\Bigg) \in \dom(T_j).
\end{equation} 
Hypothesis \eqref{2.9} then implies that $\wti \Phi_j \Phi_j f = \Phi_j f$,
which translates into 
\begin{equation}
\sum_{\ell=1}^{2 \Def(T_j)} c_{j,\ell} \wti \Phi_j \Phi_j f_{j,\ell} + \wti \Phi_j g_j
=\sum_{\ell=1}^{2 \Def(T_j)} c_{j,\ell} \Phi_j f_{j,\ell} + g_j,
\end{equation}
and hence, using \eqref{2.9} again on the terms under the sum,
\begin{equation}
\wti \Phi_j g_j = g_j, \quad j \in K. 
\end{equation}

If $j\in J \backslash  K$, we simply set
\begin{equation}
g_j=\Phi_j f \in\dom(T_j^*)=\dom(T_j),
\end{equation}
since in this case $\Def(T_j)=0$. Another straightforward application of \eqref{2.9}
then yields the rest of the cases (with $j\in J \backslash  K$) since 
\begin{equation}
\wti \Phi_j g_j = g_j, \quad j \in J.       \lb{2.40}
\end{equation}
By \eqref{2.11a} one obtains that 
\begin{equation}
g_j \in \dom(T), \quad j \in J.      \lb{2.41} 
\end{equation}

Next, by \eqref{2.14}, one concludes that $\big(I_{\cH} - \sum_{j \in J} \Phi_j\big) f \in \dom(T)$ 
and hence $f = \sum_{j \in J} \Phi_j f  \; ({\rm mod}\, \dom(T))$ implies that there exists
$g_0 \in \dom(T)$ such that 
\begin{align}
f = \sum_{j \in J} \Phi_j f + g_0   
= \sum_{k \in K} \Bigg(\sum_{\ell=1}^{2 \Def(T_k)} c_{k,\ell} \Phi_k f_{k,\ell}\Bigg) 
+ \sum_{j \in J} g_j + g_0.    \lb{2.48}
\end{align} 

If $\# (J) < \infty$, then \eqref{2.41} implies that $\sum_{j \in J} g_j  \in \dom(T)$, which in turn 
implies that $f$ and $\Phi_k f_{k,\ell}$, $1 \leq \ell \leq 2 \Def(T_k)$, $k \in K$, are linearly 
dependent $({\rm mod}\, \dom(T))$, contradicting \eqref{2.37}.

If $\# (J) = \infty$, \eqref{2.14d} implies that $T^* \big(\sum_{j \in J} \Phi_j f\big)$ is well-defined, and hence so is $T^* \big(\sum_{j \in J} \Psi_j\big)$. Since the partial sums are in $\dom(T)$, so is their limit, as $T$ is closed. Consequently, \eqref{2.18} also holds in this case. \\[.5mm]

Finally, we consider the case where 
\begin{equation} 
\sum_{j \in J} \Def(T_j) = \infty.         \lb{2.54} 
\end{equation} 
In this case, for any $N \in \N$ there exists then a finite subset $K_N\subset J$, such that for
each $k\in K_N$ there exists an integer $N_k \in \mathbb N$ with $N_k\leq\Def(T_k)$ and
\begin{equation}
\sum_{k\in K_N} N_k\geq N.
\end{equation}
(The integer $N_k$ is only needed in the situation where the corresponding $\Def(T_k)=\infty$.)
For each $k\in K_N$, let 
\begin{equation}
f_{k,\ell} \in \dom(T_{k}^*), \quad 1 \leq \ell \leq N_{k},
\end{equation}
be linearly independent $({\rm mod}\, \dom(T_{k}))$. Then, following
verbatim the first part of our proof above, one concludes again that 
$\Phi_{k} f_{k,\ell} \in \dom(T^*)$, $1\leq \ell \leq 2N_{k}$, $k \in K_N$,
are linearly independent $({\rm mod}\, \dom(T))$. Consequently, and by the choice of $K_N\subset J$,
\begin{equation}
2 \Def(T)\geq \sum_{k\in K_N} 2N_k \geq 2N.
\end{equation}
Since $N \in \N$ was arbitrary, $\Def(T) = \infty$, completing the proof of \eqref{2.18}. 
\end{proof}
%%%%%%%%%%

%%%%%%%%%%
\begin{remark} \lb{r2.11}
$(i)$ It might be surprising at first sight that our conditions \eqref{2.9}--\eqref{2.14d} (i.e., the full hypotheses of Theorem \ref{t2.10}) only involve operator domains, and do not require any further information on the operators themselves. Note however that this is consistent with our point of 
view on the deficiency index (see \eqref{defectdom}), and that, as will be shown explicitly in 
Section \ref{s5}, the conditions in Hypothesis \ref{h2.9} can all be realized very naturally due 
to locality properties of Schr\"odinger-type operators 
and second-order elliptic operators (with, possibly, strongly singular potential coefficients). 
Furthermore, while we focus exclusively on the case of 
second-order elliptic partial differential operators in Section \ref{s5}, 
the case of first-order Dirac-type operators can be discussed along analogous lines (cf.\ 
\cite{Be85}, \cite{Ka82}, \cite{Ka85}). In fact, the first-order case is technically quite a bit simpler 
than the second-order situation discussed in this paper as the analog of the term 
$- 2 (\nabla \phi_j)\cdot(\nabla f)$, see, for instance, \eqref{4.22}, and hence all the difficulties surrounding it, does not occur in the context of first-order partial differential operators. \\[.5mm]
$(ii)$ As will be illustrated in Section \ref{s5} in a fairly straightforward manner, Hypothesis \ref{h2.9} is sufficiently flexible to permit a total decoupling of singularities, with respect to their contribution to the total defect $\Def(T)$ of $T$, as long as the singularities in partial differential operators are separated by a strictly positive distance. \\[.5mm] 
$(iii)$ The abstract approach developed in this section was inspired by the concrete case of 
Dirac-type operators treated in \cite{Be85}, \cite{Jo72}, \cite{Ka82}, \cite{Ka85} (see also \cite{BG85}).  
\hfill $\diamond$
\end{remark}
%%%%%%%%%%

%%%%%%%%%%%%%%%%%%%%%%%%%
%%%%%%%%%%%%%%%%%%%%%%%%%
\section{Applications to Schr\"odinger-Type and Second-Order Elliptic Partial Differential 
Operators and Decoupling of Singularities}  \lb{s5}
%%%%%%%%%%%%%%%%%%%%%%%%%
%%%%%%%%%%%%%%%%%%%%%%%%%

In this section we apply the abstract approach developed in Theorem \ref{t2.10} to the concrete case of Schr\"odinger operators and second-order elliptic partial differential operators, each possibly containing a strongly singular potential term. Our results illustrate the concept of 
decoupling of singularities with respect to deficiency index computations whenever the singularities are separated by a fixed minimal positive distance.  

We start with the following auxiliary result. (It is surely well-known
-- the case of Lipschitz functions is covered by \cite[Theorem\ 4.12]{MMMM13} --
but we provide its proof for the convenience of the reader.)

%%%%%%%%
\begin{lemma} \lb{l4.1}
Assume $F_0, F_1 \subset \R^n$ are such that $\dist \, (F_0, F_1) \geq \varepsilon$ for some 
$\varepsilon > 0$. Then there exists a function $\phi \in C^{\infty}(\R^n)$ satisfying
\begin{align}
\begin{split} 
& 0 \leq \phi \leq 1 \text{ on } \, \R^n, 
\quad \phi\big|_{F_0} = 0, \quad \phi\big|_{F_1} = 1, \, \text{ and} \\[1mm]     
& \|\partial^\alpha\phi\|_{L^{\infty}(\R^n)} \leq c_{n,\alpha}\,\varepsilon^{-|\alpha|}
\, \text{ for each multi-index } \, \alpha \in \N_0^n,
\lb{4.1}
\end{split} 
\end{align}
where the constant $c_{n,\alpha}\in(0,\infty)$ depends only on $n$ and $\alpha$.  
\end{lemma} 
%%%%%%%%
\begin{proof}
We start by introducing 
\begin{align}
\wti F_1=\{x \in \R^n \,|\, \dist \, (x, F_1) \leq \varepsilon/4\}.
\end{align}
Consider a function 
\begin{align} 
0\leq\theta \in C_0^{\infty}(\R^n), \quad \supp \, (\theta) \subseteq B_n(0;1), \quad 
\int_{\R^n} d^n x \, \theta(x) = 1,
\end{align}
then define $\phi$ via  
\begin{align}
\phi(x)=\left(\frac{\varepsilon}{4}\right)^{-n}\int_{\wti F_1} d^n y \, 
\theta(4(x-y)/\varepsilon), \quad 
x \in \R^n. 
\end{align}
Obviously, $\phi \in C^{\infty}(\R^n)$ and for each $x \in \R^n$ one has 
\begin{align}
0 \leq\phi(x)\leq 
\left(\frac{\varepsilon}{4}\right)^{-n}\int_{\R^n} d^n y \, \theta (4(x-y)/\varepsilon)
=\int_{\R^n} d^n y \, \theta(y)=1. 
\end{align}
One observes that if $x \in F_0$, then for each $y\in\wti F_1$ one necessarily has 
$|x-y| \geq\varepsilon/4$. Since by construction $\supp \, (\theta) \subseteq B_n(0;1)$, 
this forces $\theta (4(x-y)/\varepsilon)=0$. One therefore obtains $\phi(x) = 0$ for each 
$x \in F_0$, and hence, $\phi\big|_{F_0} = 0$. Similarly, if $x \in F_1$, then 
necessarily 
\begin{align}\label{MM.2}
\supp \big(\theta (4(x-\cdot)/\varepsilon)\big)\subset B_n(x;\varepsilon/4)\subset\wti F_1.
\end{align}
Consequently, for each $x \in F_1$ one has 
\begin{align}
\phi(x)=\left(\frac{\varepsilon}{4}\right)^{-n}\int_{\R^n} d^n y \, \theta (4(x-y)/\varepsilon)
=\int_{\R^n} d^n y \, \theta(y)=1,
\end{align}
and hence, $\phi\big|_{F_1} = 1$. 

Finally, for every multi-index $\alpha$ one estimates  
\begin{align}
|(\partial^\alpha\phi)(x)| &= \left(\frac{\varepsilon}{4}\right)^{-n-|\alpha|}\bigg|
\int_{\wti F_1} d^n y \, (\partial^\alpha\theta) (4(x-y)/\varepsilon)\bigg|       
\no \\[1mm]
&\leq \left(\frac{\varepsilon}{4}\right)^{-n-|\alpha|} 
\int_{|x-y|\leq\varepsilon/4} d^n y \, |(\partial^\alpha\theta) (4(x-y)/\varepsilon)|
\no \\[1mm]
&\leq \|\partial^\alpha\theta\|_{L^\infty(\R^n)}
\left(\frac{\varepsilon}{4}\right)^{-n-|\alpha|}\int_{|y|\leq\varepsilon/4} d^n y 
\no \\[1mm]
&=c_{n,\alpha}\,\varepsilon^{-|\alpha|},\quad x\in\R^n,
\end{align} 
for some finite constant $c_{n,\alpha}>0$ depending only on $n$ and $\alpha$. 
\end{proof}
%%%%%%%%

To set up the type of Schr\"odinger operators we are interested in, we next collect the following set of assumptions.

%%%%%%%%
\begin{hypothesis} \lb{h4.2} 
Let $J \subseteq \N$ be an index set and $n \in \N$, $n \geq 2$. \\[.5mm] 
$(i)$ Consider compact sets $\Sigma_j \subset \R^n$ of $n$-dimensional Lebesgue measure zero, $j \in J$.  \\[.5mm]  
$(ii)$ Consider $V_j \in L^2_{\loc}(\R^n \backslash \Sigma_j)$ real-valued and with bounded support, $j \in J$. \\[.5mm]
$(iii)$ Suppose there exists $\varepsilon > 0$ such that 
\begin{equation}
\dist \, (\{\supp \, (V_j) \cup \Sigma_j\}, \{\supp \, (V_{j'}) \cup \Sigma_{j'}\}) \geq \varepsilon, \quad 
j, j' \in J, j \neq j'.     \lb{4.2} 
\end{equation}
\end{hypothesis}
%%%%%%%%

Granted Hypothesis \ref{h4.2}, we introduce
\begin{align}
& \Sigma = \bigcup_{j \in J} \Sigma_j,    \lb{4.3} \\
& A_j = \supp \, (V_j) \cup \Sigma_j, \; j \in J, 
\quad A = \bigcup_{j \in J} A_j = \bigcup_{j \in J} \supp \, (V_j) \cup \Sigma,    \lb{4.4} \\
& V(x) = \sum_{j \in J} V_j(x) \, \text{ for a.e.\ } \, x \in \R^n \backslash \Sigma.    \lb{4.5} 
\end{align}

One notes that $A_j$ are compact sets and, as a consequence of the uniform separation of sets, properly stated in \eqref{4.2},
that $\Sigma$ and $A$ are closed subsets of $\R^n$. In addition, $\Sigma$ is of $n$-dimensional Lebesgue measure zero, $V$
is real-valued, and, due to the uniform separation of the $A_j$'s, $V\in L^2_{\loc}(\R^n \backslash \Sigma)$.

Next, we introduce the symmetric Schr\"odinger operators in $L^2(\R^n)$,
\begin{align}
& \dot H_j f = (- \Delta f) + V_j  f, \quad f \in \dom(\dot H_j) 
= C_0^{\infty} (\R^n \backslash \Sigma_j),  
\quad j \in J,    \lb{4.8} \\
& \dot H f = (- \Delta f) + V  f, \quad f \in \dom(\dot H) 
= C_0^{\infty} (\R^n \backslash \Sigma),  
\lb{4.9} 
\end{align}
whose closures in $L^2(\R^n)$ are denoted by $H_j$, $j \in J$, and $H$, respectively, 
and whose adjoints are then given by (cf., e.g., \cite{Ka80}, \cite[Sect.\ 2.1]{RSS94})
\begin{align}
\begin{split} 
& H_j^* f = (- \Delta f) + V_j  f \, \text{ in } \, \cD (\R^n \backslash \Sigma_j)^{\prime}, \\
& f \in \dom(H_j^*) = \big\{g \in L^2(\R^n) \, \big| \, [- (\Delta g) + V_j  g] \in L^2(\R^n)\big\},  
\quad j \in J,    \lb{4.10} 
\end{split} \\
\begin{split} 
& H^* f = (- \Delta f) + V  f \, \text{ in } \,  \cD (\R^n \backslash \Sigma)^{\prime},  \\
& f \in \dom(H^*) = \big\{g \in L^2(\R^n) \, \big| \, [- (\Delta g) + V  g] \in L^2(\R^n)\big\}.  
\lb{4.11} 
\end{split} 
\end{align}
Next we will show that the abstract Theorem \ref{t2.10} applies to $H$ and $H_j$, $j \in J$, 
by proving a series of results that verify each item in Hypothesis \ref{h2.9}; in fact, we will 
typically prove slightly stronger results. Moreover, note that by Theorem~\ref{t2.8} (with
$\cC$ the standard complex conjugation of complex-valued functions) one has $\Def(H)=n_+(H)=n_-(H)$ as well as
$\Def(H_j)=n_+(H_j)=n_-(H_j)$, $j \in J$.

We start with the following auxiliary result.

%%%%%%%%%
\begin{lemma} \lb{l4.3} 
Assume Hypothesis \ref{h4.2}. There exist real-valued 
$\phi_j, \wti \phi_j \in C_0^{\infty}(\R^n)$, 
$j \in J$, such that the following conditions $(i)$--$(v)$ hold: \\[.5mm] 
$(i)$ $\partial^{\alpha} \phi_j \in L^{\infty}(\R^n)$, $0 \leq |\alpha| \leq 2$, $\phi_j \big|_{A_j} =1$, 
$j \in J$. \\[.5mm] 
$(ii)$ $\supp \, (\phi_j) \cap \supp \, (\phi_{j'}) = \emptyset$, $j, j' \in J$, $j \neq j'$. \\[.5mm] 
$(iii)$ For some $0 < \delta < \varepsilon/2$, $\dist \, (\supp \, (1 - \phi_j), A_j) \geq \delta$, $j \in J$. 
\\[.5mm] 
$(iv)$ $\partial^{\alpha} \wti \phi_j \in L^{\infty}(\R^n)$, $0 \leq |\alpha| \leq 2$, 
$\wti \phi_j \big|_{\supp \, (\phi_j)} =1$, $j \in J$. \\[.5mm] 
$(v)$ $\supp \, \big(\wti \phi_j\big) \cap \supp \, \big(\wti \phi_{j'}\big) = \emptyset$, $j, j' \in J$, 
$j \neq j'$.
\end{lemma} 
%%%%%%%%%
\begin{proof}
Fix $j \in J$ and define $U_{j,\eta} = \bigcup_{a \in A_j} B_n(a; \eta)$, $\eta > 0$. 
Then $U_{j,\varepsilon/4}$ is an open neighborhood of $A_j$ and 
\begin{equation}
E_{j,\varepsilon/4} = \ol{U_{j,\varepsilon/4}}, \quad 
F_{j,\varepsilon/4} = \R^n\backslash U_{j,\varepsilon/2}, \quad j \in J,
\end{equation}  
are closed and disjoint. By Lemma \ref{l4.1} one can find $\phi_j \in C^{\infty}(\R^n)$ such that 
\begin{equation}
\phi_j \big|_{E_{j,\varepsilon/4}} =1, \quad \phi_j \big|_{F_{j,\varepsilon/2}} = 0, \quad 
\partial^{\alpha} \phi_j \in L^{\infty}(\R^n), \; 0 \leq |\alpha| \leq 2. 
\end{equation} 
It is now clear that one can choose $\delta = \varepsilon/4$. This shows the existence of 
$\phi_j \in C_0^{\infty}(\R^n)$, $j \in J$, satisfying properties $(i)$--$(iii)$; the existence of 
$\wti \phi_j\in C_0^{\infty}(\R^n)$, $j \in J$, satisfying items $(iv)$--$(v)$ follows 
analogously.   
\end{proof}
%%%%%%%%%

In the following we identify $\Phi_j$ and $\wti \Phi_j$ with the operator of multiplication 
by the bounded, real-valued functions $\phi_j$ and $\wti \phi_j$, $j \in J$, defined on all of 
$L^2(\R^n)$, respectively.

For simplicity, we focus on the case $n \in \N$, $n \geq 2$, throughout this section. 
The case $n=1$ is obviously analogous (and by far simpler). 

The next result verifies the analogs of conditions \eqref{2.10} and \eqref{2.11} (in fact, it proves 
additional facts).

%%%%%%%%%
\begin{lemma} \lb{l4.5}
Assume Hypothesis \ref{h4.2}. Then for all $j \in J$, the following conditions $(i)$--$(ii)$ hold: \\[.5mm] 
$(i)$ $f \in \dom(H_j^*)$ implies $\phi_j f \in \dom(H_j^*) \cap \dom(H^*)$. \\[.5mm] 
$(ii)$ $f \in \dom(H^*)$ implies $\phi_j f \in \dom(H^*) \cap \dom(H_j^*)$. \\[.5mm] 
In both cases, 
\begin{equation} 
H_j^* (\phi_j f) = H^* (\phi_j f), \quad j \in J.    \lb{4.14a} 
\end{equation}
All statements also hold with $\phi_j$ replaced by $\wti \phi_j$.
\end{lemma}
%%%%%%%%%
\begin{proof} 
$(1)$ Let $f \in \dom(H_j^*)$ and $\psi_j \in C_0^{\infty}(\R^n\backslash A_j)$, $j \in J$. Then 
$H_j^* f = [(- \Delta f) + V_j f] \in L^2(\R^n)$ implies 
\begin{equation}
\psi_j (H_j^* f) = \psi_j (- \Delta f) \in  L^2(\R^n), 
\end{equation}
which together with the fact that $\psi_j \in C_0^{\infty}(\R^n\backslash A_j)$ is arbitrary, implies that 
\begin{equation}
\Delta f \in L_{\loc}^2(\R^n \backslash A_j).
\end{equation}
Thus (cf., e.g., \cite[Theorem\ 1]{Ka80}) also
\begin{equation}
\nabla f \in L_{\loc}^2(\R^n \backslash A_j)^n, 
\end{equation}
and hence 
\begin{equation}
f \in H_{\loc}^{2} (\R^n \backslash A_j), 
\end{equation} 
which in turn implies 
\begin{equation} 
\dom(H_j^*) \subseteq H_{\loc}^{2} (\R^n \backslash A_j).    \lb{4.20} 
\end{equation} 
$(2)$ Let $f \in \dom(H^*)$ and $\psi \in C_0^{\infty}(\R^n\backslash A)$. Then by precisely 
the same arguments one concludes that $\psi (H^* f) = \psi (- \Delta f) \in  L^2(\R^n)$, and 
\begin{equation}
\nabla f \in L_{\loc}^2(\R^n \backslash A)^n,  \quad 
\Delta f \in L_{\loc}^2(\R^n \backslash A), 
\end{equation}
and hence 
\begin{equation} 
\dom(H^*) \subseteq H_{\loc}^{2} (\R^n \backslash A).  
\end{equation}
$(3)$ Let $f \in \dom(H_j^*)$, then 
\begin{align}
\begin{split} 
& (- \Delta  + V_j)(\phi_j f) = (- \Delta + V)(\phi_j f)    \lb{4.22} \\ 
& \quad = \phi_j (- \Delta + V_j)f - 2 (\nabla \phi_j) \cdot (\nabla f) - (\Delta \phi_j) f.   
\end{split} 
\end{align}
Since 
\begin{equation} 
\phi_j, |\nabla \phi_j|, (\Delta \phi_j) \in L^{\infty}(\R^n), \quad 
\nabla \phi_j \in C_0^{\infty}(\R^n \backslash A)^n, 
\end{equation} 
in fact, 
\begin{equation} 
\nabla \phi_j\big|_{E_{j,\varepsilon/8}} = 0, \quad \nabla \phi_j \big|_{F_{j,\varepsilon/4}} = 0,  \lb{4.23} 
\end{equation} 
with $E_{j,\varepsilon/8} \supset A_j$, and $\nabla f \in L_{\loc}^2(\R^n \backslash A_j)^n$ 
by item $(1)$, one concludes that 
\begin{equation}
(\nabla \phi_j) \cdot (\nabla f) = 0 \, \text{ in a neighborhood of $A_j$ (in fact, of $A$).}   \lb{4.24} 
\end{equation}
Thus, 
\begin{equation}
(\nabla \phi_j) \cdot (\nabla f) \in L^2(\R^n),    \lb{4.25} 
\end{equation}
and hence 
\begin{equation}
 (- \Delta  + V_j)(\phi_j f) = (- \Delta + V)(\phi_j f) \in L^2(\R^n)     \lb{4.26} 
\end{equation}
proving item $(i)$. \\
$(4)$ Let $f \in \dom(H^*)$ be arbitrary. Then reasoning precisely along the lines in item $(3)$ one 
obtains 
\begin{align}
\begin{split} 
& (- \Delta  + V)(\phi_j f) = (- \Delta + V_j)(\phi_j f)   \\ 
& \quad = \phi_j (- \Delta + V)f - 2 (\nabla \phi_j) \cdot (\nabla f) - (\Delta \phi_j) f, \quad j \in J.   
\end{split} 
\end{align}
Since $\nabla f \in L_{\loc}^2(\R^n \backslash A)^n$, \eqref{4.23} with 
$E_{j,\varepsilon/8} \supset A_j$ once more yields $(\nabla \phi_j) \cdot (\nabla f) = 0$ in a 
neighborhood of $A$ and hence 
\begin{equation}
(\nabla \phi_j) \cdot (\nabla f) \in L^2(\R^n), \quad j \in J.   \lb{4.29} 
\end{equation}
Thus, 
\begin{equation}
 (- \Delta  + V)(\phi_j f) = (- \Delta + V_j)(\phi_j f) \in L^2(\R^n)    \lb{4.30}
\end{equation}
proving item $(ii)$. 
Equations \eqref{4.26} and \eqref{4.30} also prove \eqref{4.14a}.
\end{proof}
%%%%%%%%%

The following result verifies the analogs of conditions \eqref{2.11a} and \eqref{2.12} (again, additional 
facts are derived).

%%%%%%%%%
\begin{lemma} \lb{l4.6}
Assume Hypothesis \ref{h4.2}. Then for all $j \in J$, the following conditions $(i)$--$(ii)$ hold: 
\\[.5mm] 
$(i)$ $f \in \dom(H_j)$ implies $\phi_j f \in \dom(H_j) \cap \dom(H)$. \\[.5mm] 
$(ii)$ $f \in \dom(H)$ implies $\phi_j f \in \dom(H) \cap \dom(H_j)$. \\[.5mm] 
In both cases, 
\begin{equation} 
H_j (\phi_j f) = H (\phi_j f), \quad j \in J.    \lb{4.31} 
\end{equation}
All statements also hold with $\phi_j$ replaced by $\wti \phi_j$.
\end{lemma}
%%%%%%%%%
\begin{proof}
$(1)$ Let $f \in \dom(H_j)$. Since $H_j = \ol{\dot H_j}$, there exists a sequence 
$\{f_m\}_{m \in \N} \subset C_0^{\infty} (\R^n \backslash \Sigma_j)$ such that 
$\slim_{m\to\infty} f_m = f$ and $\slim_{m\to\infty} \dot H_j f_m = H_j f$. Consequently, 
$\phi_j f_m \in C_0^{\infty} (\R^n \backslash \Sigma_j)$, $\slim_{m\to\infty} \phi_j f_m = \phi_j f$, 
and 
\begin{align}
\begin{split} 
& \dot H_j (\phi_j f_m) = (- \Delta + V_j)(\phi_j f_m) = (- \Delta + V)(\phi_j f_m) 
= \dot H (\phi_j f_m)   \lb{4.33} \\
& \quad  = \phi_j (- \Delta + V_j) f_m - 2 (\nabla \phi_j) \cdot (\nabla f_m) - (\Delta \phi_j) f_m. 
\end{split} 
\end{align}
Clearly, 
\begin{equation}
\slim_{m\to\infty} \phi_j (- \Delta + V_j)f_m = \slim_{m\to\infty} \phi_j \dot H_j f_m 
=\phi_j H_j f, \quad \slim_{m\to\infty} (\Delta \phi_j) f_m = (\Delta \phi_j) f.     \lb{4.34} 
\end{equation}
Next, let $\psi \in C_0^{\infty}(\R^n \backslash A)$ be real-valued. Then 
$f \in H_{\loc}^{2} (\R^n \backslash A_j)$ (cf.\ \eqref{4.20}) implies that 
\begin{align}
& \int_{\R^n} d^n x \, \psi(x)^2 |\nabla (f_m (x) - f(x))|^2    \no \\
& \quad = 
- 2 \int_{\R^n} d^n x \, \ol{[f_m (x) - f(x)]} \psi(x) (\nabla \psi)(x) \cdot (\nabla (f_m - f))(x)   \no \\
& \qquad - \int_{\R^n} d^n x \, \ol{[f_m (x) - f(x)]} \psi(x)^2 (\Delta (f_m - f))(x),
\end{align}
and hence,
\begin{align}
& \|\psi |\nabla (f_m -f)| \|_{L^2(\R^n)}^2    \no \\
& \quad \leq 2 \| \nabla \psi \|_{L^{\infty}(\R^n)} 
\|f_m - f \|_{L^2(\R^n)} \| \psi |\nabla (f_m - f)| \|_{L^2(\R^n)}    \no \\
& \qquad + \| f_m - f \|_{L^2(\R^n)} \big\| \psi^2 [\Delta (f_m - f)]\big\|_{L^2(\R^n)}. 
\lb{4.36} 
\end{align}
Since $\psi \in C_0^{\infty}(\R^n \backslash A)$, one concludes that 
$\psi^2 [\Delta (f_m - f)] = \psi^2 [H_j (f_m - f)]$ and hence 
$\lim_{m\to \infty} \big\| \psi^2[\Delta(f_m - f)] \big\|_{L^2(\R^n; d^nx)} = 0$. 
Inequality \eqref{4.36} is of the type 
\begin{equation}
A_m^2 \leq c_m |A_m| + d_m, \, \text{ with } \, \lim_{m\to\infty} c_m = \lim_{m\to\infty} d_m = 0. 
\lb{4.37} 
\end{equation}
Consequently, the real sequence $\{A_m\}_{m\in\N}$ is bounded, that is, for some $C > 0$, 
$|A_m| \leq C$, and thus, actually,
\begin{equation}
\lim_{m\to\infty} A_m = 0.     \lb{4.38}
\end{equation}
Employing \eqref{4.37}, \eqref{4.38} in \eqref{4.36} then yields
\begin{equation}
\lim_{m\to\infty} \| \psi |\nabla (f_m - f)| \|_{L^2(\R^n)} = 0,
\end{equation}
and choosing $\psi = (\partial_k \phi_j)$, $1 \leq k \leq n$, implies
\begin{equation}
\lim_{m\to\infty} \| (\nabla \phi_j) \cdot (\nabla (f_m - f))\|_{L^2(\R^n))} = 0.   \lb{4.40} 
\end{equation}
Combining \eqref{4.33}, \eqref{4.34}, and \eqref{4.40}, yields
\begin{equation}
\slim_{m\to\infty} \dot H_j (\phi_j f_m) = \slim_{m\to\infty} \dot H (\phi_j f_m) 
= \phi_j H_j f - 2 (\nabla \phi_j) \cdot (\nabla f) - (\Delta \phi_j) f,     \lb{4.41}
\end{equation}
and since $H_j$ and $H$ are closed operators, one concludes 
$\phi_j f \in \dom(H_j) \cap \dom(H)$ and $H_j (\phi_j f) = H (\phi_j f)$, proving item $(i)$. \\
$(2)$ Interchanging $H_j$ and $H$, $\Sigma_j$ and $\Sigma$, $A_j$ and $A$, noticing that $j \in J$ was arbitrary 
in part $(1)$, yields item $(ii)$ along precisely the same steps.  
\end{proof}
%%%%%%%%%

The next result verifies the analogs of conditions \eqref{2.13} and \eqref{2.14}. 

%%%%%%%%%
\begin{lemma} \lb{l4.7}
Assume Hypothesis \ref{h4.2}. Then the following conditions $(i)$--$(ii)$ hold: \\[.5mm] 
$(i)$ $f \in \dom(H_j^*)$ implies $(1 - \phi_j) f \in \dom(H_j)$, $j \in J$. \\[.5mm] 
$(ii)$ $f \in \dom(H^*)$ implies $\big(1 - \sum_{j \in J} \phi_j\big) f \in \dom(H)$. 
\end{lemma}
%%%%%%%%%
\begin{proof}
$(1)$ Let $f \in \dom(H_j^*)$ be arbitrary. Then $(1 - \phi_j) f \in \dom(H_j^*)$ by Lemma \ref{l4.5}\,$(i)$. 
Hence, $(1 - \phi_j) f \in L^2(\R^n)$ and 
$H_j^* [(1 - \phi_j) f] = - \Delta [(1 - \phi_j) f] \in L^2(\R^n)$, implying 
\begin{equation}
\nabla [(1 - \phi_j) f] \in L^2(\R^n)^n,   \lb{4.42} 
\end{equation}
and hence 
\begin{equation}
(1 - \phi_j) f \in H^{2} (\R^n).     \lb{4.43} 
\end{equation}
(To verify the claims \eqref{4.42} and \eqref{4.43} it suffices to employ the standard Fourier transform $\cF$ in $L^2(\R^n)$, and denoting $\widehat u = \cF u$, one concludes that 
$\big(1 + |\xi|^2\big) \widehat u \in L^2(\R^n; d^n \xi)$ implies  
$|\xi| \, \widehat u \in L^2(\R^n; d^n \xi)$, $\xi_{\ell} \, \xi_m \, \widehat u \in L^2(\R^n; d^n \xi)$, 
$1 \leq \ell, m \leq n$.) 

Next, we recall that for $\Omega_k \subseteq \R^n$ open, $k=1,2$, with 
$\Omega_1 \subset \Omega_2$, 
$f \in H^{m,2}(\Omega_2)$ implies that $f\big|_{\Omega_1} \in H^{m,2}(\Omega_1)$, 
$m \in \N \cup \{0\}$ (see, e.g., \cite[p.\ 253--254]{RS78}). Hence, using the fact that 
$\dist \, (\supp \, (1 - \phi_j), A_j) \geq \varepsilon/8$, one concludes that 
\begin{equation}
(1 - \phi_j) f \in H^{2}(\R^n) \, \text{ implies } \, 
(1 - \phi_j) f \big|_{\R^n \backslash A_j} \in H^{2}(\R^n \backslash A_j),  
\end{equation}
and an application of \cite[Theorem\ V.3.4]{EE89} yields 
\begin{equation}
\eta_m (1 - \phi_j) f \in H_0^{2} (\R^n \backslash A_j),    \lb{4.55} 
\end{equation}
where 
\begin{equation} 
\eta_m \in C_0^{\infty} (\R^n), \quad 0 \leq \eta_m \leq 1, 
\quad \eta_m(x) = 1, \; |x| \leq m, \; m \in \N,      \lb{4.56}
\end{equation}  
is a suitable cutoff function. Having established \eqref{4.55}, one concludes the existence of a 
sequence $\{g_m\}_{m\in\N} \in C_0^{\infty} (\R^n \backslash A_j)$ such that 
$g_m \underset{m \to \infty}{\longrightarrow} (1 - \phi_j) f$ in 
$H^{2} (\R^n \backslash A_j)$-norm. Consequently, $\slim_{m\to\infty} g_m = (1 - \phi_j) f$ 
and  
\begin{equation}
\slim_{m\to\infty} \dot H_j g_m = \slim_{m\to\infty} (- \Delta g_m) = - \Delta [(1 - \phi_j) f] 
= H_j^* [(1 - \phi_j) f] \in L^2(\R^n),    \lb{4.57} 
\end{equation}
implying 
\begin{equation} 
H_j [(1 - \phi_j) f] = H_j^* [(1 - \phi_j) f] \in L^2(\R^n),     \lb{4.58} 
\end{equation} 
since $H_j = \ol{\dot H_j}$ 
is a closed operator. This proves item $(i)$. \\
$(2)$ Replacing $H_j^*$ by $H^*$, $\dot H_j$ by $\dot H$, $A_j$ by $A$, $(1 - \phi_j)$ by 
$\big(1 - \sum_{j \in J} \phi_j\big)$, and noticing that $j \in J$ was arbitrary in step $(1)$, one 
can now follow the strategy of proof above line by line to arrive at a proof of item $(ii)$.  
\end{proof}
%%%%%%%%%

For variants of Lemmas \ref{l4.5}--\ref{l4.7} (with somewhat different proofs) we refer to \cite{BG85}, 
\cite{CG78}, \cite[Sects.\ VII.2, VII.3]{EE89}, \cite{Kn78}, \cite{Pe75}, and \cite{Pe60}. These 
results clearly demonstrate the local nature of the operators $H_j$, $H_j^*$, $j \in J$, $H$, 
and $H^*$ (cf.\ also \cite[Sect.\ 2.5]{Am81}, \cite[Sects.\ 6.4, 10.2]{Pe88}).

Finally, we now prove that also condition \eqref{2.14d} holds in the present context of Schr\"odinger 
operators.

%%%%%%%%%
\begin{lemma} \lb{l4.8}
Assume Hypothesis \ref{h4.2}. If $f \in \dom(H^*)$, then 
\begin{equation}
\slim_{N\to\infty} H^* \bigg(\sum_{\substack{j \in J \\ |j| \leq N}} \phi_j f \bigg) 
= H^* \bigg(\sum_{j \in J} \phi_j f \bigg).   
\end{equation}
\end{lemma}
%%%%%%%%%
\begin{proof}
Given $f \in \dom(H^*)$, it suffices to write 
\begin{align}
& H^* \bigg(\sum_{\substack{j \in J \\ |j| \leq N}} \phi_j f \bigg) = 
\bigg(\sum_{\substack{j \in J \\ |j| \leq N}} \phi_j \bigg) H^* f 
- 2 \sum_{\substack{j \in J \\ |j| \leq N}} (\nabla \phi_j) \cdot (\nabla f)  
- \bigg(\sum_{\substack{j \in J \\ |j| \leq N}} (\Delta \phi_j) \bigg) f,    \lb{4.59} 
\end{align}
noticing that $\nabla \phi_j \in C_0^{\infty} (\R^n \backslash A)^n$, 
$\nabla f \in L^2_{\loc}(\R^n \backslash A)^n$, and 
$\supp \, (\phi_j) \cap \supp \, (\phi_k) = \emptyset$ for $j, k \in J$, $j \neq k$, and hence,   
\begin{align}
\begin{split} 
& \slim_{N\to\infty} \sum_{\substack{j \in J \\ |j| \leq N}} \phi_j = \sum_{j \in J} \phi_j, 
\quad \slim_{N\to\infty} \sum_{\substack{j \in J \\ |j| \leq N}} (\partial_k \phi_j) 
= \sum_{j \in J} (\partial_k \phi_j),  \; 1 \leq k \leq n,    \lb{4.59a} \\
& \slim_{N\to\infty} \sum_{\substack{j \in J \\ |j| \leq N}} (\Delta \phi_j) 
= \sum_{j \in J} (\Delta \phi_j),    
\end{split} 
\end{align}
concluding the proof. 
\end{proof}
%%%%%%%%%

Combining Lemmas \ref{l4.3}--\ref{l4.8} then shows that all items in 
Hypothesis \ref{h2.9} are satisfied and hence Theorem \ref{t2.10} yields the following result 
for Schr\"odinger operators with a possibly strongly singular potential: 

%%%%%%%%%%
\begin{theorem} \lb{t4.9}
Assume Hypothesis \ref{h4.2}. Then 
\begin{equation}
\Def\,(H) = \sum_{j \in J} \Def(H_j),     \lb{4.60} 
\end{equation} 
including the possibility that one, and hence both sides of \eqref{4.60} equal $\infty$. 
\end{theorem}
%%%%%%%%%%  

%%%%%%%%%
\begin{remark}\label{R:bounded}
The statement of Theorem \ref{t4.9} remains valid if one adds $L^\infty$ potentials to the singular potentials $V$
and/or $V_j$, $j\in J$. Indeed, this follows directly from the stability of the deficiency indices under perturbations (see, 
for example, \cite{BF77}). This will be particularly relevant in the discussion on Example \ref{e4.12}.
${}$ \hfill $\diamond$ 
\end{remark}
%%%%%%%%%

Next, we extend Theorem \ref{t4.9} to more general second-order elliptic partial differential operators in $L^2_w (\R^n)$ with possibly strongly singular potential coefficients as follows. 
For simplicity, we will again focus on the case $n \in \N$, $n \geq 2$, only. 

%%%%%%%%
\begin{hypothesis} \lb{h4.10} 
Let $n \in \N$, $n \geq 2$. 
In addition to Hypothesis \ref{h4.2}, assume the following conditions on the coefficients
$w$, $a_{k,\ell}$, $b_k$, $1\leq k, \ell \leq n$, and $V_j$, $j \in J$. \\[.5mm]
$(i)$ $w \in L^{\infty}_{\loc}(\R^n)$, $w > 0$~a.e.~on $\R^n$, and 
$w^{-1} \in L^{\infty}_{\loc}(\R^n)$. \\[.5mm] 
$(ii)$ $a_{k,\ell} = a_{\ell,k} \in C^2(\R^n)$ are real-valued, 
$1\leq k, \ell \leq n$ and $a = \{a_{j,k}\}_{1 \leq j, k \leq n}$ satisfies the local uniform ellipticity condition,  
\begin{equation}
\sum_{1 \leq k, \ell \leq n} a_{k,\ell}(x)\xi_k \xi_{\ell} \geq \lambda(x) \|\xi\|_{\R^n}^2 
\quad x, \xi \in \R^n, 
\end{equation}
where $\lambda>0$ and continuous on $\R^n$. \\[.5mm]
$(iii)$ $b_k \in C^1(\R^n)$ are real-valued, $1 \leq k \leq n$. \\[.5mm] 
$(iv)$ For fixed $\varepsilon_0 > 0$ and $\alpha_j \in [\varepsilon_0,1]$, 
$V_j \in Q_{\alpha_j, \loc}(\R^n)$, $j \in J$. \\[.5mm]
$(v)$ Suppose the maximally defined operator in $L^2_w  (\R^n)$ generated by the differential expression 
\begin{equation}
L_0 (a,b) := \f{1}{w(x)}\sum_{1 \leq k,\ell \leq n} D_k a_{k,\ell}(x) D_{\ell}, \quad 
D_k = i \partial_k + b_k(x), \; 1 \leq k \leq n, \quad x \in \R^n,  
\end{equation}
is self-adjoint, and essentially self-adjoint on $C_0^{\infty}(\R^n)$, 
more precisely, assume that 
\begin{equation}
\dot H_0 (a,b) f = L_0 (a,b) f, \quad f \in \dom(\dot H_0(a,b)) = C_0^{\infty} (\R^n), 
\end{equation}
is essentially self-adjoint in $L^2_w  (\R^n)$, and its closure, denoted by $H_0 (a,b)$, 
\begin{align}
\begin{split}
& H_0 (a,b) f = \ol{\dot H_0 (a,b)} f = L_0 (a,b) f, \\ 
& f \in \dom(H_0 (a,b)) = \big\{g \in L^2_w  (\R^n) \, \big| \, g \in H^2_{\loc}(\R^n); \, 
L_0(a,b) g \in L^2_w  (\R^n)\big\}, 
\end{split} 
\end{align}
is self-adjoint in $L^2_w  (\R^n)$.  
\end{hypothesis}
%%%%%%%%

In this context, the local Stummel space $Q_{\alpha, \loc}(\R^n)$, $\alpha \in (0,1]$ is 
defined by 
\begin{equation}
Q_{\alpha, \loc}(\R^n) := \begin{cases} L^2_{\loc}(\R^n), \quad n=2,3, \\[2mm] 
\Big\{W \in L^2_{\loc}(\R^n) \,\Big| \, \text{for all $K \in \R^n$ compact, there exists} \\ 
\hspace*{2.6cm} \text{$C_{W, K} > 0$ such that for all $x \in K$,} \\ 
\hspace*{.3cm} \int_{K \cap \ol{B_n (x;1)}} d^n y \, |x-y|^{4 - n - \alpha} |W(y)|^2 \leq C_{W,f}\Big\}, 
\quad n \geq 4.
\end{cases}    
\end{equation} 

Since the principal target in this paper are strongly singular (electric) potentials $V_j$, 
$j \in J$, $V$, we introduced  
Hypothesis \ref{h4.10}\,$(v)$ for the second-order part $- \f{1}{w}\sum_{1 \leq k, \ell \leq n} D_k a_{j,k} D_{\ell}$. For a variety of sufficient conditions on $a_{k,\ell}$, $b_k$, 
$1 \leq k, \ell \leq n$, for Hypothesis \ref{h4.10}\,$(v)$ to hold, we refer, for instance, to 
\cite{St66}. 

Given Hypotheses \ref{h4.2} and \ref{h4.10}, we now introduce in analogy to 
\eqref{4.8}, \eqref{4.9}, the following symmetric operators in $L^2_w  (\R^n)$, 
\begin{align}
& \dot H_j (a,b) f = L_0 (a,b) f + (V_j + V_{0,j}) f, \quad f \in \dom(\dot H_j (a,b)) 
= C_0^{\infty} (\R^n \backslash \Sigma_j),  
\quad j \in J,    \lb{4.68} \\
& \dot H (a,b) f = L_0 (a,b) f + (V + V_0) f, \quad f \in \dom(\dot H (a,b)) 
= C_0^{\infty} (\R^n \backslash \Sigma),  
\lb{4.69} 
\end{align}
whose closures in $L^2_w  (\R^n)$ are denoted by $H_j (a,b)$, $j \in J$, and $H (a,b)$, respectively. Their adjoints are then given by (cf.\ \cite{St66})
\begin{align}
& H_j (a,b)^* f = L_0 (a,b) f + (V_j + V_{0,j}) f \, \text{ in } \, 
\cD (\R^n \backslash \Sigma_j)^{\prime},   \no \\
& f \in \dom(H_j (a,b)^*) = \big\{g \in L^2_w  (\R^n) \, \big| \, 
g \in H^{2}_{\loc} \R^n \backslash \Sigma_j); \,   \lb{4.70}   \\
& \hspace*{3.35cm} [L_0 (a,b) g + (V_j + V_{0,j}) g] \in L^2_w  (\R^n)\big\},  
\quad j \in J,    \no \\
& H (a,b)^* f = L_0 (a,b) f + (V + V_0) f \, \text{ in } \, 
\cD (\R^n \backslash \Sigma)^{\prime},  \no \\
& f \in \dom(H (a,b)^*) = \big\{g \in L^2_w  (\R^n) \, \big| \, 
g \in H^{2}_{\loc} \R^n \backslash \Sigma); \,    \lb{4.71} \\
& \hspace*{3.25cm} [L_0 (a,b) g + (V + V_0) g] \in L^2_w  (\R^n)\big\}.   \no 
\end{align}  
Repeatedly employing product identities of the type 
\begin{align}
L_0(a,b) (\psi f) &= \psi L_0(a,b) f 
+ \f{2i}{w}\sum_{1\leq k, \ell \leq n} (\partial_k \psi) a_{k,\ell} (D_{\ell} f)   \lb{4.73} \\
& \quad - \f{1}{w}\sum_{1 \leq k, \ell \leq n} (\partial_k a_{k,\ell} \partial_{\ell} \psi) f, 
\quad f, \psi \in C^2(\R^n),   \no 
\end{align}
and closely following the arguments leading to Theorem \ref{t4.9}, the analog of the latter now 
reads as follows:

%%%%%%%%%%
\begin{theorem} \lb{t4.11}
Assume Hypothesis \ref{h4.10}. Then 
\begin{equation}
\Def\,(H(a,b)) = \sum_{j \in J} \Def(H_j(a,b)),     \lb{4.74} 
\end{equation} 
including the possibility that one, and hence both sides of \eqref{4.74} equal $\infty$. 
\end{theorem}
%%%%%%%%%%  
\begin{proof}
It suffices to sketch the necessary modifications in the proofs of Lemmas \ref{l4.5}--\ref{l4.8}, 
replacing $H_j$ by $H_j(a,b)$, $j \in J$, and $H$ by $H(a,b)$, respectively.  
\\[.5mm] 
{\it Lemma \ref{l4.5}}: Due to the fact $g \in H^{2}_{\loc} (\R^n \backslash \Sigma_j)$, $j\in J$,  
respectively, $g \in H^{2}_{\loc} (\R^n \backslash \Sigma)$, items $(1)$ and $(2)$ are clear from 
the outset. The analog of identity \eqref{4.22} now reads
\begin{align}
& \f{1}{w}\bigg(\sum_{1 \leq k, \ell \leq n} D_k a_{j,k} D_{\ell} + V_j\bigg)(\phi_j f) 
= \f{1}{w}\bigg(\sum_{1 \leq k, \ell \leq n} D_k a_{j,k} D_{\ell} + V\bigg)(\phi_j f)    
\no \\ 
& \quad = \phi_j \f{1}{w}\bigg(\sum_{1 \leq k, \ell \leq n} D_k a_{j,k} D_{\ell} + V_j\bigg)f 
+ \f{2i}{w}\sum_{1\leq k, \ell \leq n} (\partial_k \phi_j) a_{k,\ell} (D_{\ell} f)   \lb{4.75} \\
& \qquad - \f{1}{w}\sum_{1 \leq k, \ell \leq n} (\partial_k a_{k,\ell} \partial_{\ell} \phi_j) f, 
\quad f \in \dom(H_j(a,b)^*), \; j \in J,   \no 
\end{align}
and the rest of item $(3)$ proceeds along the same lines. Item $(4)$ follows in the same manner.
\\[.5mm]
{\it Lemma \ref{l4.6}}: One can follow the arguments in \eqref{4.31}--\eqref{4.41} line by line, 
replacing the term $(\nabla \phi_j) \cdot (\nabla f_m -f)$ by 
$\f{- i}{w}\sum_{1\leq k, \ell \leq n} (\partial_k \phi_j) a_{k,\ell} D_{\ell} (f_m - f)$, once 
more choosing $\psi = (\partial_k \phi_j)$, and observing that 
$(\partial_k \phi_j) \in C_0^{\infty}(\R^n \backslash A)$ (which neutralizes the effect of 
$b_k \in C^1(\R^n)$ in $D_k$), $1 \leq k \leq n$. \\[.5mm]
{\it Lemma \ref{l4.7}}: Once more employing the fact that 
$g \in H^{2}_{\loc} (\R^n \backslash \Sigma_j)$, $j\in J$ (resp, 
$g \in H^{2}_{\loc} (\R^n \backslash \Sigma)$), one again obtains the crucial inclusion 
\eqref{4.55}, and then concludes the argument along the lines of \eqref{4.56}--\eqref{4.58}, 
substituting $\f{1}{w}\sum_{1 \leq k,\ell \leq n} D_k a_{k,\ell}(x) D_{\ell}$ for 
$-\Delta$ in \eqref{4.57}.  \\[.5mm]
{\it Lemma \ref{l4.8}}: Replacing identity \eqref{4.59} by
\begin{align}
H(a)^* \bigg(\sum_{\substack{j \in J \\ |j| \leq N}} \phi_j f \bigg) &= 
\bigg(\sum_{\substack{j \in J \\ |j| \leq N}} \phi_j \bigg) H(a)^* f 
+ \f{2i}{w}\sum_{\substack{j \in J \\ |j| \leq N}} \sum_{1 \leq k, \ell \leq n} (\partial_k \phi_j) 
a_{k,\ell} (D_{\ell} f)    \no \\
& \quad - \f{1}{w}\sum_{\substack{j \in J \\ |j| \leq N}} \sum_{1 \leq k, \ell \leq n} 
(\partial_k a_{k,\ell} \partial_{\ell} \phi_j) f,  \quad f \in \dom(H(a)^*), 
\end{align}
one proceeds along the lines leading up to  \eqref{4.59a}, replacing the latter by 
\begin{align}
\begin{split} 
& \slim_{N\to\infty} \sum_{\substack{j \in J \\ |j| \leq N}} \phi_j = \sum_{j \in J} \phi_j, 
\quad \slim_{N\to\infty} \f{1}{w}\sum_{\substack{j \in J \\ |j| \leq N}} (\partial_k \phi_j) 
= \f{1}{w}\sum_{j \in J} (\partial_k \phi_j),  \; 1 \leq k \leq n,   \lb{4.76} \\
& \slim_{N\to\infty} \f{1}{w}\sum_{\substack{j \in J \\ |j| \leq N}} 
\sum_{1 \leq k, \ell \leq n}(\partial_k a_{k,\ell} \partial_{\ell} \phi_j) 
= \f{1}{w}\sum_{j \in J} \sum_{1 \leq k, \ell \leq n}(\partial_k a_{k,\ell} \partial_{\ell} \phi_j),   
\end{split} 
\end{align} 
completing the proof. 
\end{proof}
%%%%%%%%%%

We conclude this section illustrating the scope of Theorems \ref{t4.9} and \ref{t4.11} in 
connection with Schr\"odinger-type operators $H$ as well as second-order elliptic partial 
differential operators $H(a)$:

%%%%%%%%%%
\begin{example} \lb{e4.12}
Let $J \subseteq \N$ be an index set. \\[.5mm]
$(i)$ Assume that the set of points $\{x_j\}_{j \in J} \subset \R^n$ satisfies 
\begin{equation}
\inf_{\substack{j, j' \in J \\ j \neq j'}} |x_j - x_{j'}| > 0.    \lb{4.77}
\end{equation} 
Then concrete examples of potential coefficients $V$ with strong point-like singularities 
in $H$ and $H(a,b)$ are, for instance, given by
\begin{align}
\begin{split} 
V(x) := \sum_{j \in J} V_j(x) = \sum_{j \in J} c_j\left(\frac{x-x_j}{|x-x_j|}\right) |x - x_j|^{- \alpha_j} \chi_{B_n(x_j;\varepsilon/4)}(x),& \\
x \in \R^n \backslash \{x_j\}_{j \in J},&    \lb{4.78} 
\end{split} 
\end{align}
where 
\begin{align}
\alpha_j\ge 0, \quad  c_j \in L^{\infty}(S^{n-1}) \, \text{ real-valued}, \quad j \in J.    \lb{4.79}
\end{align}
In this case, 
\begin{equation} 
\Sigma_j = \{x_j\}, \quad j \in J, \quad  
\Sigma = \bigcup_{j \in J} \Sigma_j = \{x_j\}_{j \in J}, 
\end{equation} 
and the potential $V_j$ comprises the $j$th term on the right-hand side of \eqref{4.78}. \\[.5mm]
$(ii)$ As an example with strong shell-like singularities in $H$ and $H(a,b)$ we mention, for 
instance,  
\begin{align}
\begin{split} 
V(x)  \equiv \sum_{j \in J} V_j(x)  
=  \sum_{j \in J} \beta_j | |x - x_j| - r_j|^{- \gamma_j}\chi_{B_n(x_j;\delta_j)}(x),&     \lb{4.81} \\
x \in \R^n \Big\backslash \bigcup_{j \in J} \{x \in \R^n \, | \, |x - x_j| = r_j\},&     
\end{split} 
\end{align}
where
\begin{align}
& 0<r_j<\delta_j <\varepsilon/2, \quad  \beta_j \in \R, \quad \gamma_j \geq 0,\quad j \in J,     \\
& \{x \in \R^n \, | \, |x - x_j| \leq r_j\} \cap \{x \in \R^n \, | \, |x - x_{j'}| \leq r_{j'}\} = \emptyset, 
\quad \; j \neq j', \; j, j' \in J    \no
\end{align}
$($e.g., for some $\eta \in (0,1/4)$, $r_j \leq \varepsilon \eta$, $j \in J$\,$)$. In this case, 
\begin{align} 
\begin{split} 
\Sigma_j &= \{x \in \R^n \, | \, |x - x_j| = r_j\}, \quad j \in J,      \\ 
\Sigma &= \bigcup_{j \in J} \Sigma_j = \bigcup_{j \in J} \{x \in \R^n \, | \, |x - x_j| = r_j\},
\end{split}
\end{align}  
and $V_j$ comprises the $j$th term on the right-hand side of \eqref{4.81}. 
\end{example} 
%%%%%%%%%%

These examples clearly illustrate the notion of decoupling of singularities when computing 
deficiency indices as long as all singularities are separated by a minimal positive distance. 

As a concrete example, we conclude this section with the proof of
Theorem \ref{t1.1} :

%%%%%%%%
\begin{proof}[Proof of Theorem \ref{t1.1}]
Let $\delta>0$ be arbitrary, and note that, for each $j\in J$,
one can write, 
\begin{equation}
V_j(|x-x_j|)\chi_{B_n(x_j;\delta)}(x)=V_{\loc,j}(x)+V_{0,j}(x),
\end{equation}
with
\begin{align}
& V_{\loc,j}(x)=V_j(|x-x_j|) \chi_{B_n(x_j;\varepsilon/2)}(x),   \no \\
& V_{0,j}(x)= V_j(|x-x_j|) \chi_{B_n(x_j;\delta) \backslash  B_n(x_j;\varepsilon/2)}(x), 
\quad x \in \R^n \backslash\{x_j\}.    \lb{5.82}
\end{align}
Since the supports of the functions $V_{0,j}$ form a locally finite family
of sets, and the functions $V_{0,j}$ are bounded uniformly in both $x\in\R^n$
and $j\in J$, one concludes that their sum is well-defined and bounded,
\begin{equation} 
\wti{V}_0= V_0+ \sum_{j\in J} V_{0,j}\in L^\infty(\R^n).
\end{equation} 
Thus the potential function $V$ from \eqref{E:defnV_multipole} can be written as 
\begin{equation}
V= \wti{V}_0+\sum_{j\in J} V_{\loc,j},
\end{equation}
and so, by the main result in \cite{BF77},
\begin{equation}\label{E:loc1}
\Def(H)=\Def(H_\loc), 
\end{equation}
with the ``localized'' operator $H_\loc$ being 
\begin{equation}
H_\loc = -\Delta+\sum_{j\in J} V_{\loc,j},  \quad 
\dom(H_\loc)=C_0^\infty(\R^n \backslash \{x_j\}_{j\in J}).
\end{equation}
At this point one can apply Theorem \ref{t4.9} since Hypothesis \ref{h4.2} is satisfied with  
$\Sigma_j=\{x_j\}$ for each $j\in J$, and the singular potentials $V_{\loc,j}$ as defined in 
\eqref{5.82}. Thus, one concludes that
\begin{equation}\label{E:loc2}
\Def (H_\loc)=\sum_{j\in J} \Def (H_{\loc,j}), 
\end{equation}
with
\begin{equation}
H_{\loc,j} = -\Delta+ V_{\loc,j}, \quad 
\dom(H_{\loc,j})=C_0^\infty(\R^n \backslash \{x_j\}).
\end{equation}
Finally, we note again that for each $j\in J$, one has 
$\Def (H_{\loc,j}) =0$ if and only if \eqref{E:cond_c_j} holds.
Indeed, since $H_{\loc,j}$ commutes with rotations one can use separation of variables
in spherical coordinates and $\Def (H_{\loc,j}) =0$ if and only if this holds
for each angular momentum operator, which in turn holds if and only if
\begin{equation}
\frac{(n-1)(n-3)}{4} + c_j \ge \frac{3}{4}
\end{equation}
by \cite[Theorem~2.4]{KST10}. Thus, 
\begin{equation}
\Def(H)=\sum_{j\in J}\Def(H_j),
\end{equation}
finishing the proof of Theorem \ref{t1.1}.
\end{proof}
%%%%%%%%

As explained in the Introduction, it was this example and 
the expectation that uniformly separated singularities of the 
potential (cf.\ \eqref{4.77}) decouple in the context of deficiency index computations 
that motivated our interest in this circle of ideas.

%%%%%%%%%%%%%%%%%%%%%%%%%%%%%%%%%%%%%%%%
%%%%%%%%%%%%%%%  The Appendix  %%%%%%%%%%%%%%%%%
%%%%%%%%%%%%%%%%%%%%%%%%%%%%%%%%%%%%%%%%

\begin{appendix}

%%%%%%%%%%%%%%%%%%%%%%%%%%%%%%%%%%%%%%%%
%%%%%%%%%%%%%%%%%%%%%%%%%%%%%%%%%%%%%%%%
\section{The Support of an Arbitrary Function Defined in an Arbitrary Subset 
of $\R^n$}    \lb{sA}
%%%%%%%%%%%%%%%%%%%%%%%%%%%%%%%%%%%%%%%%
%%%%%%%%%%%%%%%%%%%%%%%%%%%%%%%%%%%%%%%%

In this appendix we provide a discussion of the notion of support for arbitrary functions on 
arbitrary subsets of $\R^n$. 

%%%%%%%%%%%
\begin{definition}\label{spt-arb}
Given an arbitrary set $E\subseteq\R^n$ and an arbitrary
function $f:E\to\C\cup\{\infty\}$, define the support of $f$ as the set
\begin{equation}\label{PJE-1XV.4R-1}
\supp \, (f):=\big\{x\in E\,|\,\text{there is no $r>0$ so that $f=0$ 
a.e.\ in $B_n(x,r)\cap E$}\big\},
\end{equation}
where ``a.e." is interpreted with respect to the $n$-dimensional Lebesgue measure 
in $\R^n$. 
\end{definition}
%%%%%%%%%%%

In addition, given an arbitrary set $E\subseteq\R^n$ introduce 
\begin{align}\label{PJ-Fa.1}
\begin{split} 
E^+:=\{x\in\R^n\,|\, &\text{there is no $r>0$ so 
that $B_n(x,r)\cap E$ is contained}    \\
& \; \text{in a set of $n$-dimensional Lebesgue measure zero}\},
\end{split} 
\end{align}
and observe that 
\begin{align}\label{PJ-Fa.Ya.1}
& \mathring E \subseteq E^+\subseteq\overline{E},
\\[1mm]
& E^+=\overline{E}\,\,\text{ if $E$ is open}.
\label{PJ-Fa.Ya.2}
\end{align}

Throughout, if $A\subseteq\R^n$ is measurable, we denote by $|A|$ its 
$n$-dimensional Lebesgue measure.

%%%%%%%%%%%
\begin{lemma}\label{SuPT}
For an arbitrary set $E\subseteq\R^n$ and two arbitrary
functions $f,g:E\to\C\cup\{\infty\}$, the following properties hold:
\begin{align}\label{spt-0}
& \text{$E^+$ is a closed subset of $\R^n$ and $|E\backslash  E^+|=0$},
\\[.5mm]
&\text{for each $F\subseteq E$ the function $\chi_F:E\to\C$ satisfies } \,
\supp \, (\chi_F) = F^+ \cap E,
\label{spt-aaa}
\\[.5mm]
& \supp \, (f) \text{ is a relatively closed subset of $E^+\cap E$ $($hence of $E$\,$)$},
\label{spt-1}
\\[.5mm]
& f=0 \text{ a.e.\ in } \, E\bs \supp \, (f),
\label{spt-2}
\\[.5mm]
& \supp \, (f) \subseteq F \text{ if $F$ relatively 
closed subset of $E$ and $f=0$ a.e on $E \backslash F$},
\label{spt-3}
\\[.5mm]
& \supp \, \big(f\big|_{F}\big)\subseteq F^+\cap F\cap\supp \, (f) \text{ for each } F\subseteq E,
\label{spt-3bb}
\\[.5mm]
& \supp \, (f)=\supp \, (g) \,\text{ if $f=g$ a.e.\ on $E$},
\label{spt-4}
\\[.5mm]
& \supp \, (fg)\subseteq \supp \, (f) \cap \supp \, (g),
\label{spt-5}
\\[.5mm]
& \supp \, (f+g)\subseteq \supp \, (f) \cup \supp \, (g).
\label{spt-6}
\end{align}

In addition, if the set $E\subseteq\R^n$ is open and the 
function $f:E\to\C$ is continuous, then $\supp \, (f)$ may be 
described as the relative closure in $E$ of the set $\{x\in E \,|\, f(x)\not=0\}$, which is 
precisely the standard notion of support in this context. 

Furthermore, if $E\subseteq\R^n$ is open and if $f\in L^1_{\loc}(E)$ then 
one has $\supp \, (u_f) = \supp \, (f)$, where $u_f$ is the distribution canonically 
associated with the locally integrable function $f$ in the open set $E$.
\end{lemma}
%%%%%%%%%%%
\begin{proof}
By \eqref{PJ-Fa.1} and the fact that the Lebesgue measure is complete,  
it follows that for each point $x\in\R^n\backslash  E^+$ there exists 
some number $r_x>0$ such that $|B_n(x,r_x)\cap E|=0$. We claim that for the family 
$\{r_x\}_{x\in\R^n\backslash  E^+}$ as above, one has  
\begin{equation}\label{supEEE}
\R^n\backslash  E^+=
\bigcup\limits_{x\in\R^n\backslash  E^+}B_n(x,r_x).
\end{equation}
The left-to-right inclusion is obvious. To justify the opposite one, one notes that
if $y$ belongs to the right-hand side of \eqref{supEEE}, then 
$y\in\R^n$ and there exists some $x\in\R^n$ such that 
$|B_n(x,r_x)\cap E|=0$ and $y\in B_n(x,r_x)$. Consequently, choosing $r:=r_x-|x-y|>0$ 
forces $B_n(y,r) \subseteq B_n(x,r_x)$, hence $|B_n(y,r)\cap E|=0$. 
This shows that $y\notin E^+$, concluding the proof of \eqref{supEEE}. 
In turn, \eqref{supEEE} implies that $\R^n\backslash  E^+$ is open, 
thus $E^+$ is a closed subset of $\R^n$. This takes care of 
the first claim in \eqref{spt-0}. 

Next, one observes that since $\R^n$ is a strongly Lindel\"of space
(as a  second-countable topological space), the union in the right-hand side 
of \eqref{supEEE} may be refined to a countable one. 
Thus, one can find a sequence of points 
$\{x_j\}_{j\in{\mathbb{N}}}\subset \R^n\backslash  E^+$ 
along with a sequence of numbers $\{r_j\}_{j\in{\mathbb{N}}}\subset (0,\infty)$ 
such that   
\begin{align}\label{supp-ee.1}
& |B_n(x_j,r_j)\cap E|=0 \,\text{ for each } \,j\in{\mathbb{N}}, 
\\[6pt]
& \text{and } \,
E\backslash  E^+=E\cap \bigg(\bigcup\limits_{j\in{\mathbb{N}}}B_n(x_j,r_j)\bigg).
\label{supp-ee.2}
\end{align}
As such, the last claim in \eqref{spt-0} readily follows 
from \eqref{supp-ee.1}--\eqref{supp-ee.2}.

As far as \eqref{spt-aaa} is concerned, pick some $F\subseteq E$. Then 
$x\in E\backslash \supp \, (\chi_F)$ if and only if $x\in E$ and 
there exists $r>0$ such that $\chi_F=0$ a.e.\ on $B_n(x,r)\cap E$. 
Since $\chi_F=1$ on $B_n(x,r)\cap F\subseteq B_n(x,r)\cap E$, the latter condition 
is further equivalent to $|B_n(x,r)\cap F|=0$, which ultimately is 
equivalent to $x\notin F^+$. This reasoning shows that 
$E\backslash \supp \, (\chi_F) = E\backslash  F^+$, so by passing to complements 
(relative to $E$, and keeping in mind that $\supp \, (\chi_F) \subseteq E$) 
one obtains $\supp \, (\chi_F) = E\backslash (E\backslash  F^+)=E\cap F^+$,
concluding the proof of \eqref{spt-aaa}.

Next, we note that by design, 
\begin{equation}\label{support-1aa}
\supp \, (f) \subseteq E^+\cap E.
\end{equation}
Also, for each $x\in(E^+\cap E)\backslash \supp \, (f)$ there exists a number $r_x>0$ 
such that $f=0$ a.e.\ in $B_n(x,r_x)\cap E$, and we claim that 
\begin{equation}\label{support-1}
(E^+\cap E)\backslash \supp \, (f) = (E^+\cap E)\cap \bigg(
\bigcup\limits_{x\in(E^+\cap E)\backslash \supp \, (f)}B_n(x,r_x) \bigg).
\end{equation}
Indeed, the left-to-right inclusion is tautological, so we focus on the opposite one.
In this regard, if $y$ belongs to the right-hand side of \eqref{support-1}, then 
$y\in E^+\cap E$ and there exist a point $x\in E^+\cap E$ such that $f=0$ a.e.\ 
in $B_n(x,r_x)\cap E$ and $y\in B_n(x,r_x)$. 
Then, if $r:=r_x-|x-y|>0$, it follows that $B_n(y,r)\subseteq B_n(x,r_x)$, hence 
$f=0$ a.e.\ in $B_n(y,r)\cap E$. This shows that $y\notin\supp \, (f)$, finishing 
the proof of \eqref{support-1}. In turn, from \eqref{support-1aa} and \eqref{support-1}
(and \eqref{spt-0}) one deduces that \eqref{spt-1} holds. 

As before, based on the fact that $\R^n$ is a strongly Lindel\"of space, 
one can find a sequence of points 
$\{x_j\}_{j\in{\mathbb{N}}} \subset (E^+\cap E)\backslash \supp \, (f)$ 
and a sequence of numbers $\{r_j\}_{j\in{\mathbb{N}}} \subset (0,\infty)$ 
such that   
\begin{equation}\label{support-1bb}
 f=0 \text{ a.e.\ in } \,B_n(x_j,r_j)\cap E \,\text{ for each } \, j \in \N,  
\end{equation} 
and 
\begin{equation} 
(E^+\cap E)\backslash \supp \, (f) = (E^+\cap E)\cap
\bigg(\bigcup\limits_{j\in{\mathbb{N}}}B_n(x_j,r_j)\bigg).
\label{support-1cc}
\end{equation}
Thus, \eqref{spt-2} readily follows from \eqref{support-1bb}--\eqref{support-1cc}
and \eqref{spt-0}.

Next, assume that $F$ is a relatively closed subset of $E$ with the property 
that $f=0$ a.e.\ on $E\backslash  F$, and pick an arbitrary point $x\in E\backslash  F$.
Given that $E\backslash  F$ is relatively open in $E$, it follows that there exists 
a number $r>0$ such that $B_n(x,r)\cap E\subseteq E\backslash  F$. This implies  
$f=0$ a.e.\ on $B_n(x,r)\cap E$ which, in turn, implies $x\notin\supp \, (f)$. 
Thus, $\supp \, (f) \subseteq F$, establishing \eqref{spt-3}.
Finally, \eqref{spt-3bb} is readily implied by \eqref{spt-1} and \eqref{PJE-1XV.4R-1},
while properties \eqref{spt-4}--\eqref{spt-6} are seen directly from \eqref{PJE-1XV.4R-1}.

Suppose now that $E\subseteq\R^n$ is open.
Then \eqref{PJE-1XV.4R-1} yields
\begin{align}\label{Pvd-1}
\begin{split} 
E\backslash \supp \, (f) &= \{x\in E\,|\, \text{there exists $r>0$ so that 
$B_n(x,r)\subseteq E$}  \\
& \hspace*{3.8cm} \text{and } f=0 \text{ a.e.\ in } B_n(x,r)\}. 
\end{split} 
\end{align}
In particular, if $f\in L^1_{\loc}(E)$, from \eqref{Pvd-1} one obtains 
$E\backslash \supp \, (f) = E\backslash \supp \, (u_f)$, where $u_f$ is the 
distribution canonically associated with the function $f$ in the open set $E$.
If in addition $f$ is continuous, then \eqref{Pvd-1} further becomes
\begin{align}\label{ga.tGav}
E\backslash \supp \, (f) &= \{x\in E \,|\,\text{there exists $r>0$ so that $B_n(x,r)\subseteq E$}  \no \\ 
& \hspace*{4.5cm} \text{and $f=0$ in $B_n(x,r)$} \}     \no \\ 
&=E\backslash \overline{\{x\in E\,|\, f(x)\not=0\}},
\end{align}
as was to be shown.
\end{proof}
%%%%%%%%%%%

\end{appendix}
%%%%%%%%%%%%%%%%%%%%%%%%%%%%%%%%%%%%%%%%
%%%%%%%%%%%%%%%%%%%%%%%%%%%%%%%%%%%%%%%%

\medskip 

\noindent 
%%%%%%%%%%%%%%%%%%%%%%%%%%%%%%%%%%%%%
{\bf Acknowledgments.} We are indebted to George Hagedorn, Werner Kirsch, 
Roger Nichols, and Peter Pfeifer for helpful discussions.  

F.G.\ and I.N.\ gratefully acknowledges a kind invitation to the Faculty of Mathematics, University of Vienna, Austria, for parts of June 2014. The extraordinary hospitality, as well as the stimulating atmosphere at the department, are greatly appreciated. 
%%%%%%%%%%%%%%%%%%%%%%%%%%%%%%%%%%%%%

%%%%%%%%%%%%%%%%%%%%%%%%%%%%%%%%
%%%%%%%%%%%%%%%%%%%%%%%%%%%%%%%%t

\end{document}